\documentclass[12pt,reqno]{amsart}
\usepackage{amsmath}
\usepackage{amsthm}
\usepackage{amsfonts}
\usepackage{amssymb}
\usepackage{mathrsfs}
\usepackage[margin=1.35in, top =1.25in, bottom = 1.25in]{geometry}
\usepackage{setspace}
\usepackage{tikz}
\usetikzlibrary{matrix, arrows}
\usepackage{tikz-cd}
\usepackage{enumitem}
\usepackage[colorlinks, breaklinks]{hyperref}
\hypersetup{linkcolor=black,citecolor=blue,filecolor=black,urlcolor=black} 
\usepackage{pdflscape}
\usepackage{array}
\usepackage{adjustbox}
\usepackage{tabularx}
\usepackage{multirow}
\usepackage{multicol}
\usepackage{float}
\usepackage{arydshln}
\usepackage{appendix}

\renewcommand{\a}{\mathfrak{a}}
\DeclareMathOperator{\Ann}{Ann}

\DeclareMathOperator{\ch}{ch}				
\DeclareMathOperator{\coker}{Coker}
\newcommand{\dd}{\partial}
\DeclareMathOperator{\Ext}{Ext}
\DeclareMathOperator{\hgt}{ht}  			
\DeclareMathOperator{\Id}{Id}				
\DeclareMathOperator{\image}{Im}
\renewcommand{\Im}{\image}
\newcommand{\iso}{\cong}
\newcommand{\kk}{k}					
\DeclareMathOperator{\ld}{ld}				
\newcommand{\longto}{\longrightarrow}
\DeclareMathOperator{\pd}{pd}  			
\renewcommand{\phi}{\varphi}
\DeclareMathOperator{\rank}{rank}
\DeclareMathOperator{\reg}{reg}			
\DeclareMathOperator{\Sym}{Sym}
\DeclareMathOperator{\syz}{Syz} 	
\newcommand{\T}{\mathsf{T}}				
\DeclareMathOperator{\Tor}{Tor}
\DeclareMathOperator{\unmixed}{unm}

\newtheorem{thm}{Theorem}[section]
\newtheorem{lemma}[thm]{Lemma}

\newtheorem{prop}[thm]{Proposition}
\newtheorem{cor}[thm]{Corollary}

\newtheorem{mainthm}{Theorem}

\theoremstyle{definition}

\newtheorem{example}[thm]{Example}
\newtheorem{rmk}[thm]{Remark}
\newtheorem{question}[thm]{Question}

\newtheorem*{notation}{Notation}
\newtheorem*{ack}{Acknowledgements}

\numberwithin{equation}{section}
\numberwithin{table}{section}

\title{The Structure of Koszul Algebras Defined by Four Quadrics}
\author{Paolo Mantero and Matthew Mastroeni}
\date{}

\begin{document}

\begin{abstract}
Avramov, Conca, and Iyengar ask whether $\beta_i^S(R) \leq \binom{g}{i}$ for all $i$ when $R=S/I$ is a Koszul algebra minimally defined by $g$ quadrics.  In recent work, we give an affirmative answer to this question when $g \leq 4$ by completely classifying the possible Betti tables of Koszul algebras defined by height-two ideals of four quadrics.  Continuing this work, the current paper proves a structure theorem for Koszul algebras defined by four quadrics. We show that all these Koszul algebras are LG-quadratic, proving that an example of Conca of a Koszul algebra that is not LG-quadratic is minimal in terms of number of defining equations.  We then characterize precisely when these rings are absolutely Koszul, and establish the equivalence of the absolutely Koszul and Backelin--Roos properties up to field extensions for such rings (in characteristic zero).  The combination of the above paper with the current one provides a fairly complete picture of all Koszul algebras defined by $g \leq 4$ quadrics.
\end{abstract}

\maketitle

\begin{spacing}{1.1}
\section{Introduction}

A graded ideal $I$ in a standard graded polynomial ring $S$ over a field $\kk$ defines a \emph{Koszul algebra} $R = S/I$ if $\kk \cong R/R_+$ has a linear free resolution over $R$.  This extraordinary homological condition leads to numerous other restrictions on the Hilbert series and Betti numbers of $R$ over $S$; see \cite{Froberg:Koszul:algebras:survey} and \cite{Koszul:algebras:and:their:syzygies}.  In particular, a question of Avramov, Conca, and Iyengar asks:  

\begin{question}[{\cite[6.5]{free:resolutions:over:Koszul:algebras}}] \label{Betti:number:bound:for:Koszul:algebras}
If $R$ is Koszul and $I$ is minimally generated by $g$ quadrics, does the following inequality hold for all $i$?
\[ \beta^S_i(R) \leq \binom{g}{i} \]
In particular, is $\pd_S R \leq g$?
\end{question}

This question is known to have an affirmative answer when $R$ is LG-quadratic.  We say that $R$ or $I$ is \emph{G-quadratic} if, after a suitable linear change of coordinates $\phi: S \to S$, the ideal $\phi(I)$ has a Gr\"obner basis consisting of quadrics.  We also say that $R$ or $I$ is \emph{LG-quadratic} if $R$ is a quotient of a G-quadratic algebra $A$ by an $A$-sequence of linear forms.  Every G-quadratic algebra is Koszul by upper semicontinuity of the Betti numbers; see \cite[3.13]{upper:semicontinuity}.  Since deforming by a regular sequence of linear forms preserves Koszulness (see Proposition \ref{passing:Koszulness:to:and:from:quotients}), every LG-quadratic algebra is also Koszul. 

Almost all known examples of (commutative) Koszul algebras are LG-quadratic.  They include all quotients by quadratic monomial ideals, all quadratic complete intersections \cite[1.2.5]{Caviglia:PhD:thesis}, the coordinate rings of Grassmannians in their Pl\"ucker embedding \cite{Grassmannians:are:G-quadratic} and most canonical curves \cite{Grobner:flags}, many types of toric rings \cite{Hibi:rings} \cite{normal:polytopes:and:triangulations} \cite{Koszul:bipartite:edge:rings}, and all suitably high Veronese subrings of any standard graded algebra \cite{Eisenbud:Reeves:Totaro}. However, there is one notable exception due to Conca.

\begin{example}[{\cite[3.8]{Koszul:algebras:and:their:syzygies}}] \label{Koszul:but:not:LG-quadratic}
The ring 
\[ R = \kk[x,y,z,w]/(xy, xw, (x-y)z, z^2, x^2+zw), \] 
defined by $g = 5$ quadrics, is Koszul by a filtration argument, but it is not LG-quadratic since the numerator of its Hilbert series 
\[ H_R(t) = \frac{1+2t-2t^2-2t^3+2t^4}{(1-t)^2} \] 
cannot be realized by any 5-generated edge ideal.  Nonetheless, the Betti table of $R$ is
\begin{center}
   \begin{tabular}{c|cccccccc}
  & 0 & 1 & 2 & 3 & 4 \\ 
\hline 
0 & 1 & -- & --  & -- & -- \\ 
1 & -- & 5 & 4  & -- & -- \\
2 & -- & -- & 4 & 6 & 2 \\
\end{tabular}
\vspace{0.2 cm}
\end{center}
so Question \ref{Betti:number:bound:for:Koszul:algebras} still has an affirmative answer in this case.  \end{example}

In general, Question \ref{Betti:number:bound:for:Koszul:algebras} is known to have an affirmative answer for Koszul algebras defined by $g \leq 3$ quadrics \cite[4.5]{Koszul:algebras:defined:by:3:quadrics} and for Koszul almost complete intersections (where $\hgt I = g -1$) with any number of generators \cite{Koszul:ACI's}.  In \cite{Koszul:algebras:defined:by:4:quadrics}, we give a strong affirmative answer to Question \ref{Betti:number:bound:for:Koszul:algebras} for $g = 4$ quadrics by determining the possible Betti tables of Koszul algebras defined by height two ideals minimally generated by four quadrics.  Recording the \emph{graded Betti number} $\beta_{i,i+j}^S(R) = \dim_\kk \Tor_i^S(R, \kk)_{i +j}$ in column $i$ and row $j$ of each table and representing zero entries by ``$-$'' for readability, they are:

\begin{thm}[{\cite[4.7]{Koszul:algebras:defined:by:4:quadrics}}] \label{Koszul:4:quadric:height:2:Betti:tables}
Let $R = S/I$ be a Koszul algebra defined by four quadrics with $\hgt I = 2$.  Then the Betti table of $R$ over $S$ is one of the following:
\vspace{1 ex}
\begin{center}
\begin{minipage}{\textwidth}
\begin{multicols}{2}
\begin{enumerate}[label = \textnormal{(\roman*)}]
\item 
\textnormal{
\begin{tabular}{c|cccccccc}
  & 0 & 1 & 2 & 3  \\ 
\hline 
0 & 1 & -- & --  & --\\ 
1 & -- & 4 & 4  & 1 
\end{tabular}
}

\vspace{1 ex}

\item
\textnormal{
\begin{tabular}{c|cccccccc}
  & 0 & 1 & 2 & 3 & 4 \\ 
\hline 
0 & 1 & -- & --  & -- & -- \\ 
1 & -- & 4 & 3  & 1  & -- \\
2 & -- & -- & 3 & 3 & 1
\end{tabular}
}

\item
\textnormal{
\begin{tabular}{c|cccccccc}
  & 0 & 1 & 2 & 3  \\ 
\hline 
0 & 1 & -- & --  & -- \\ 
1 & -- & 4 & 3  & -- \\
2 & -- & -- & 1 & 1 
\end{tabular}
}

\vspace{1 ex}

\item
\textnormal{
\begin{tabular}{c|cccccccc}
  & 0 & 1 & 2 & 3 & 4 \\ 
\hline 
0 & 1 & -- & --  & -- & -- \\ 
1 & -- & 4 & 2  & --  & -- \\
2 & -- & -- & 4 & 4 & 1
\end{tabular}
}
\end{enumerate}
\end{multicols}
\end{minipage}
\end{center}
\vspace{1 ex}
\end{thm}

The current paper is the natural continuation of the above work.  Our main result (Theorem \ref{thm:main}) is a structure theorem for the height two ideals generated by four quadrics defining Koszul algebras over an algebraically closed field.  Combining it with previous structure theorems for Koszul almost complete intersections with any number of generators \cite{Koszul:ACI's} yields a complete description of the structure of Koszul algebras defined by $g \leq 4$ quadrics.  We illustrate this when $g = 4$ in the theorem below.

\begin{mainthm} \label{structure:theorem}
Let $S$ be a polynomial ring over an algebraically closed field $\kk$ and $I \subseteq S$ be an ideal generated by $g = 4$ quadrics.  Then $R = S/I$ is Koszul if and only if $I$ has one of the following forms:
\begin{enumerate}[label = \textnormal{(\arabic*)}]
\item $(a_1x, \dots, a_4x)$ for some linear forms $x, a_1, \dots, a_4$ with $\hgt (a_1,\dots, a_4) = 4$ when $\hgt I = 1$. 
\item When $\hgt I = 2$:
\begin{enumerate}[label = \textnormal{(\roman*)}]
\item $(x, y) \cap (z, w)$ or $(x, y)^2 + (xz + yw)$ for independent linear forms $x$, $y$, $z$ and $w$, or $(xy, xz, xw, q)$ for independent linear forms $y$, $z$, and $w$ and some linear form $x$ and quadric $q \in (y, z, w) \setminus (x)$.
\item $(xy, xz, xw, q)$ for independent linear forms $y$, $z$, and $w$ and some linear form $x$ and quadric $q$ which is a nonzerodivisor modulo $(xy, xz, xw)$.
\item $(xz, yz, a_3x+b_3y, a_4x+b_4y)$ for some linear forms $x$, $y$, $z$, $a_3$, $a_4$, $b_3$, and $b_4$ such that $\hgt (x, y) = \hgt (a_3x+b_3y, a_4x+b_4y) = 2$ and $\hgt (z, a_3x+b_3y, a_4x+b_4y, a_3b_4-a_4b_3) = 3$.
\item $(a_1x, a_2x, b_3y, b_4y)$ for some linear forms $x$, $y$, $a_i$, $b_i$ such that $(a_1x,a_2x)$ and $(b_3y, b_4y)$ are transversal ideals with $\hgt(a_1, a_2) = \hgt(b_3, b_4) = 2$.
\end{enumerate}

\item When $\hgt I = 3$:
\begin{enumerate}[label = \textnormal{(\roman*)}]
\item $I = (xz, xw, q_3, q_4)$, where $x, z, w$ are linear forms and $q_3, q_4$ are quadrics forming a regular sequence on $S/(xz, xw)$;
\item $I = I_2(M)+(q_4)$ where $M$ is a $3 \times 2$ matrix of linear forms of $S$ with $\hgt I_2(M) = 2$ and $q_4$ is a quadric regular on $S/I_2(M)$.
\end{enumerate}

\item $I$ is a complete intersection of quadrics when $\hgt I = 4$.

\end{enumerate}

\end{mainthm}

We emphasize that Theorem \ref{structure:theorem} applies over an algebraically closed field of arbitrary characteristic.  It agrees with the computer-aided computations of Roos in \cite{quadratic:algebras:in:4:variables} for characteristic zero.  Out of the 83 different types of quadratic algebras in 4 variables identified by Roos, there are only four types of Koszul algebras defined by height-two ideals of four quadrics with different Koszul homology algebra structure (Cases 27, 25, 26, 21 in Table 4).  These cases correspond precisely to our cases (i)--(iv) respectively.  

In addition to the intrinsic interest in having a structure theorem, a large part of our interest in the investigation of Koszul algebras defined by four quadrics owes to Conca's ring $R$ in Example \ref{Koszul:but:not:LG-quadratic}, which points to unexpected complications in trying to answer Question \ref{Betti:number:bound:for:Koszul:algebras} for Koszul algebras defined by $g \geq 5$ quadrics.  It is natural to ask how minimal this example is.  Since $e(R) = 1$ and every height-one ideal of quadrics is easily seen to be G-quadratic, Conca's example is minimal in terms of multiplicity and codimension.  As a consequence of our results, we show that Conca's example is also minimal in terms of the number of quadrics.

\begin{mainthm} \label{LG-quadratic}
If $R = S/I$ is a Koszul algebra such that $I$ is minimally generated by $g \leq 4$ quadrics, then, after passing to the algebraic closure of the ground field,
$R$ is LG-quadratic.
\end{mainthm}

Hence, in hindsight, it is no surprise that Question \ref{Betti:number:bound:for:Koszul:algebras} has an affirmative answer for $g \leq 4$.  We do not know of an example of a Koszul algebra which is not LG-quadratic but which becomes LG-quadratic after extending the ground field.

Another question of interest raised in \cite{absolutely:Koszul:algebras} asks whether the absolutely Koszul and Backelin--Roos properties coincide over algebraically closed fields.  As a second application of our main theorem, we determine which Koszul algebras defined by at most four equations have the absolutely Koszul property and prove that, in our setting, the question of Conca et al. has a positive answer (at least if $\ch(\kk) = 0$). 

Finally, we spend a few words on the techniques employed in this paper.  To prove the structure theorem, we exploit a number of considerations about syzygies, linkage, annihilators of certain $\Ext$ modules, and 1-generic matrices. Deformations to ``generic" cases as well as specializations play an important role in every section, including the appendix.  In Theorem \ref{2:linear:syzygies:case}, we prove that there are four homologically distinct types of ideals having Betti table (iv) in {Theorem \ref{Koszul:4:quadric:height:2:Betti:tables}}.  However, three of them do {\em not} define Koszul algebras.  Since these cases are numerically indistinguishable from a Koszul algebra by their Betti table, we show that two of these cases are not Koszul using a recently proved condition on the first syzygy map of a Koszul algebra \cite[2.8]{Koszul:ACI's}.  Unfortunately, proving the remaining case is not Koszul is much more challenging -- indeed, the entire appendix is devoted to it. To achieve this goal, we first deform and specialize to a specific algebra which can be realized as a symmetric algebra of a certain module over a smaller ring.  Exploiting the natural bigrading and invoking a result of \cite{HHO} about algebra retracts then allows us to significantly reduce the size of the computation needed to show that this ring is not Koszul.  Thus, we think that the proofs that these algebras are not Koszul may be of independent interest.

\begin{notation} 
Throughout the remainder of the paper, the following notation will be in force unless specifically stated otherwise.  Let $\kk$ be a fixed algebraically closed ground field of arbitrary characteristic, $S$ be a standard graded polynomial ring over $\kk$,  $I \subseteq S$ be a proper graded ideal, and $R = S/I$.  Recall that the ideal $I$ is called \emph{nondegenerate}\index{nondegenerate} if it does not contain any linear forms.  We can always reduce to a presentation for $R$ with $I$ nondegenerate by killing a basis for the linear forms contained in $I$, and we will assume that this is the case throughout.  We denote the irrelevant ideal of $R$ by $R_+ = \bigoplus_{n \geq 1} R_n$.
\end{notation}

The division of the rest of the paper is as follows.  After reviewing some useful background about LG-quadratic algebras in \S \ref{background}, we complete our classification of Koszul algebras defined by four quadrics in \S \ref{multiplicity:1:case:section}.  Finally, we apply our structure theorems in \S \ref{Backelin-Roos:property:section} to determine which Koszul algebras defined by at most 4 quadrics have the Backelin-Roos property.

\begin{ack}
Evidence for the results in this paper was provided by several computations in Macaulay2 \cite{Macaulay2}. We also thank Aldo Conca for helpful conversations related to this work.
\end{ack}

\section{Background on LG-Quadratic Algebras}
\label{background}

It is useful to know how the Koszul property can be passed to and from quotient rings.

\begin{prop}[{\cite[\S 3.1, 2]{Koszul:algebras:and:regularity}}] \label{passing:Koszulness:to:and:from:quotients}
Let $S$ be a standard graded $\kk$-algebra and $R$ be a quotient ring of $S$.
\begin{enumerate}[label = \textnormal{(\alph*)}]
\item If $S$ is Koszul and $\reg_S(R) \leq 1$, then $R$ is Koszul.
\item If $R$ is Koszul and $\reg_S(R)$ is finite, then $S$ is Koszul.
\end{enumerate}
\end{prop}

Here, the \emph{regularity} of $R$ over $S$ is defined by 
\begin{equation} \label{regularity}
\reg_S(R) = \max\{j \mid \beta_{i,i+j}^S(R) \neq 0 \; \text{for some}\; i \}.
\end{equation}
We will primarily use the above proposition to show that the Koszul property can be passed back and forth between a ring and its quotient by a regular sequence consisting of linear forms or quadrics.

The following simple proposition about Hilbert series and its corollary will be useful for determining that various rings are LG-quadratic.

\begin{prop}[{\cite[3.2]{Hilbert:functions:of:graded:algebras}}]
Let $I \subseteq S$ be an ideal and $L = (x_1, \dots x_c) \subseteq S$ be an ideal generated by linear forms.  Then
\[ H_{S/(I, L)}(t) \geq (1-t)^cH_{S/I}(t), \]
with equality if and only if the generators of $L$ form a regular sequence mod $I$.
\end{prop}

\begin{cor} \label{Betti:table:preserved:implies:quotient:by:regular:sequence}
Let $I \subseteq S$ be an ideal and $L = (x_1, \dots x_c) \subseteq S$ be an ideal generated by linear forms.  If the graded Betti numbers of $S/I$ over $S$ are the same as the graded Betti numbers of $S/(I, L)$ over $S/L$, then $L$ is generated by a regular sequence mod $I$.
\end{cor}

\begin{proof}
Since the graded Betti numbers of $S/I$ determine the numerator $P(t)$ of its Hilbert series, we see that $(1-t)^cH_{S/I}(t) = (1-t)^c\frac{P(t)}{(1-t)^{\dim S}} = \frac{P(t)}{(1-t)^{\dim S/L}} = H_{S/(I, L)}(t)$ so the conclusion follows from the preceding proposition.
\end{proof}

As a particular application of this fact, we have the following proposition.  Recall that two ideals $I_1, I_2 \subseteq S$ in a standard graded polynomial ring over $\kk$ are called \emph{transversal} if $I_1 \cap I_2 = I_1I_2$.  This is easily seen to be equivalent to having $\Tor_1^S(S/I_1, S/I_2) = 0$ and, hence, $\Tor_i^S(S/I_1, S/I_2) = 0$ for all $i \geq 1$ by the rigidity of Tor over regular rings \cite[2.1]{rigidity:of:Tor}.

\begin{prop} \label{transveral:height:1:ideals:are:LG-quadratic}
Suppose that $R = S/I$ where $I = I_1 + I_2$ for transversal height one ideals $I_1$ and $I_2$ generated by quadrics.  Then $R$ is LG-quadratic.
\end{prop}

\begin{proof}
We note that $I_1 = (a_1x, \dots, a_sx)$ for some linear form $x$ and independent linear forms $a_i$.  Likewise, $I_2 = (b_1y, \dots, b_ry)$ for some linear form $y$ and independent linear forms $b_j$.  The minimal free resolutions of $S/I_1$ and $S/I_2$  are just the Koszul complexes of $a_1, \dots, a_s$ and $b_1, \dots, b_r$ respectively, except that the first differential is multiplied by $x$ or $y$ accordingly.  Furthermore, since $I_1$ and $I_2$ are transversal, the minimal free resolution of $R$ is just the tensor product of these resolutions.  Thus, the graded Betti numbers of $R$ are uniquely determined by the fact that $I_1$ and $I_2$ are transversal height one ideals of quadrics.

In the case, where $x, y, a_1, \dots, a_s, b_1, \dots, b_r$ are independent linear forms, $R$ is clearly G-quadratic, since after a suitable linear change of coordinates we may assume that the given linear forms are variables, and hence, $I$ is a monomial ideal.  In general, we can consider the G-quadratic algebra $A = \tilde{S}/\tilde{I}$ where $\tilde{S} = S[z, w, u_1, \dots, u_s, v_1, \dots, v_r]$ and $\tilde{I} = (u_1z, \dots, u_sz, v_1w, \dots, v_rw)$.  If $L = (u_i - a_i, v_j - b_j, z - x, w - y \mid 1 \leq i \leq r, 1 \leq j \leq s)$ denotes the ideal of linear forms specializing $\tilde{I}$ to $I$, we see that the graded Betti numbers of $\tilde{S}/(\tilde{I}, L) \iso S/I$ over $\tilde{S}/L \iso S$ are the same as those of $\tilde{S}/\tilde{I}$ over $\tilde{S}$.  Thus, $L$ is generated by a regular sequence mod $I$ by preceding corollary, and $R$ is LG-quadratic.
\end{proof}

\section{The Structure Theorem}
\label{multiplicity:1:case:section}

Most of this section is devoted to proving the following structure theorem.

\begin{thm}\label{thm:main}
Let $I \subseteq S$ be an ideal of height 2 generated by $g = 4$ quadrics. Then $R = S/I$ is Koszul if and only if $I$ has one of the following forms:
\begin{enumerate}[label = \textnormal{(\roman*)}]
\item $(x, y) \cap (z, w)$ or $(x, y)^2 + (xz + yw)$ for independent linear forms $x$, $y$, $z$ and $w$, or $(xy, xz, xw, q)$ for independent linear forms $y$, $z$, and $w$ and some linear form $x$ and quadric $q \in (y, z, w) \setminus (x)$;
\item $(xy, xz, xw, q)$ for independent linear forms $y$, $z$, and $w$ and some linear form $x$ and quadric $q$ which is a nonzerodivisor modulo $(xy, xz, xw)$;
\item $(xz, yz, a_3x+b_3y, a_4x+b_4y)$ for some linear forms $x$, $y$, $z$, $a_3$, $a_4$, $b_3$, and $b_4$ such that $\hgt (x, y) = \hgt (a_3x+b_3y, a_4x+b_4y) = 2$ and $\hgt (z, a_3x+b_3y, a_4x+b_4y, a_3b_4-a_4b_3) = 3$;
\item $(a_1x, a_2x, b_3y, b_4y)$ for some linear forms $x$, $y$, $a_i$, $b_i$ such that $(a_1x,a_2x)$ and $(b_3y, b_4y)$ are transversal ideals with $\hgt(a_1, a_2) = \hgt(b_3, b_4) = 2$.
\end{enumerate}
\end{thm}

We summarize the proof below.

\begin{proof}
($\Longrightarrow$): If $R$ is Koszul, then $R$ must have one of the Betti tables listed in Theorem \ref{Koszul:4:quadric:height:2:Betti:tables}. Tables (i) and (ii) correspond precisely to cases (i) and (ii) respectively by \cite[3.1, 3.3]{Koszul:algebras:defined:by:4:quadrics}.  Case (iii) follows by Theorem \ref{1:linear:syzygy:case}. Case (iv) follows by Theorem \ref{2:linear:syzygies:case} and Lemma \ref{non-Koszul:2:linear:syzygies:case}.

($\Longleftarrow$): If $I$ has form (i) or (ii), then $I$ is G-quadratic by \cite[3.5]{Koszul:algebras:defined:by:4:quadrics}. If $I$ has form (iii) or (iv), then $I$ is LG-quadratic by Corollary \ref{1:linear:syzygy:case:is:LG-quadratic} and Proposition \ref{transveral:height:1:ideals:are:LG-quadratic} respectively.
\end{proof}

As noted above, it remains only to prove structure theorems for the ideals realizing Betti tables (iii) and (iv) in Theorem \ref{Koszul:4:quadric:height:2:Betti:tables}, both of which correspond to ideals of multiplicity one.  To aid in the analysis of Betti table (iii), we make use of the following result, which is credited to Eisenbud and Evans in \cite{direct:methods:for:primary:decomposition} but for which we were unable to find an explicit reference.  Luckily, it is also essentially a result of Schenzel on the annihilators of local cohomology. 

\begin{lemma}\label{3.2}
If $R = S/I$ is a standard graded $\kk$-algebra with $\dim S = d$ and we set $\a_i = \Ann_S \Ext^i_S(R, S)$ for each $i$, then $\a_0 \cdots \a_d \subseteq I$. 
\end{lemma}

\begin{proof}
By Grothendieck duality, we know that $\Ext^i_S(R, S)$ is Matlis dual to $H^{d-i}_{R_+}(R)$ so that $\a_i = \Ann_S H^{d-i}_{R_+}(R)$ for every $i$.  Setting $\bar{\a}_i = \a_i/I$ for each $i$, it follows from \cite[Prop 3]{dualizing:complexes:and:sop's} that $\bar{\a}_0 \cdots \bar{\a}_d = 0$.  Although Schenzel's result is stated in the local case, we can reduce to the local case since the ideal $\bar{\a}_0 \cdots \bar{\a}_d$ is graded and, therefore, it is zero if and only if its localization at the irrelevant ideal is zero.  Hence, we must have $\a_0 \cdots \a_d \subseteq I$ as wanted.
\end{proof}

We will also need the following; the interested reader should consult the original version for a more general statement. We recall that the \emph{unmixed part} of an ideal $I$ of height $c$ is the intersection of all the components of height $c$ appearing in an irredundant primary decomposition of $I$ and is denoted by $I^{\unmixed}$.

\begin{prop}[{\cite[1.1]{direct:methods:for:primary:decomposition}}] \label{annihilators:of:cohomology}
Let $R = S/I$ be a standard graded $\kk$-algebra with $\hgt I = c$, and set $\a_i = \Ann_S \Ext^i_S(R, S)$ for each $i$.  Then $\hgt \a_i \geq i$ for all $i$, and $\a_c = I^{\unmixed}$.
\end{prop}

\begin{thm} \label{1:linear:syzygy:case}
The ring $R = S/I$ has Betti table 
\begin{center}
\textnormal{
\begin{tabular}{c|cccccccc}
  & 0 & 1 & 2 & 3  \\ 
\hline 
0 & 1 & -- & --  & -- \\ 
1 & -- & 4 & 3  & -- \\
2 & -- & -- & 1 & 1 
\end{tabular}
}
\end{center}
if and only if $I = (xz, yz, a_3x+b_3y, a_4x + b_4y)$ for some linear forms $x$, $y$, $z$, $a_3$, $a_4$, $b_3$, $b_4$ such that $\hgt(x, y) = \hgt (a_3x+b_3y, a_4x+b_4y)=2$, and $\hgt (z, a_3x+b_3y, a_4x+b_4y, a_3b_4-a_4b_3) = 3$. 
\end{thm}

\begin{proof}
By computing the Hilbert series from the Betti table of $R$, we see that $e(R)=1$ and $\hgt I = 2$.  It then follows from the associativity formula for multiplicity \cite[4.7.8]{Bruns:Herzog} that $I$ is contained in a unique minimal prime of height two, which has the form $P = (x, y)$ for independent linear forms $x$ and $y$, and the length of $R_P$ is one so that $P$ must be the $P$-primary component of $I$. Setting $\a_i = \Ann_S \Ext^i_S(R, S)$ for each $i$, it follows from the above proposition that $P = I^{\unmixed} = \a_2$.  By the above lemma, we also know that $\a_2\a_3 \subseteq I$ since $\a_i = R$ for $i < 2$ as $\Ext^i_S(R, S) = 0$.  Now, consider the minimal free resolution $F_\bullet$ of $R$ over $S$, which by assumption has the form
\begin{equation} \label{case:2:minimal:free:resolution}
F_\bullet: \qquad 0 \longto S(-5) \stackrel{\phi_3}{\longto} S(-3)^3 \oplus S(-4) \stackrel{\phi_2}{\longto} S(-2)^4 \stackrel{\phi_1}{\longto} S.
\end{equation}
Since $\phi_3^*: S(3)^3 \oplus S(4) \to S(5)$ is a presentation for $\Ext^3_S(R, S)$, we see that $\a_3 = I_1(\phi_3)$.  Hence, there is a linear form $z$ such that $Pz \subseteq \a_2\a_3 \subseteq  I$.  We note that $z \neq 0$; otherwise $\a_3$ would be a complete intersection of 3 quadrics by the above proposition, and at least one of the rows of the matrix of linear syzygies of $I$ would be a nontrivial linear syzygy of $\a_3$, which is impossible.  Consequently, $xz$ and $yz$ are independent quadrics so that they are part of a minimal set of generators for $I$, and we can write $I = (xz, yz, q_3, q_4)$ where $q_i = a_ix + b_iy$ for some linear forms $a_i$ and $b_i$ as $I \subseteq P$.

We claim that $\hgt (q_3,q_4) = 2$. Indeed, if this is not the case, then $(q_3,q_4)\subseteq (\ell)$ for a linear form $\ell$.  If $\ell \notin P$, then $(\ell, z)$ is a prime ideal of height at most 2 different from $P$ that contains $I$, which is impossible.  If instead $\ell \in P$, then for any linear form $y'$ with $P = (\ell, y')$, we have $I = (\ell z, y'z, q_3, q_4)$.  Thus, three of the generators of $I$ are divisible by $\ell$, which implies $e(R) = 2$ by \cite[3.3]{Koszul:algebras:defined:by:4:quadrics}. 

Set $\Delta = a_3b_4 - a_4b_3$.  It remains to show that $\hgt (z, q_3, q_4, \Delta) = 3$.  We claim that the maps $\phi_i$ in \eqref{case:2:minimal:free:resolution} can be chosen so that
\begin{equation} \label{case:2:differentials}
\phi_3 = \begin{pmatrix} \Delta \\ -q_4 \\ q_3 \\ z \end{pmatrix} 
\quad 
\phi_2 = \begin{pmatrix} 
y & a_3 & a_4 & 0 \\
-x & b_3 & b_4 & 0 \\
0 & -z & 0 & q_4 \\
0 & 0 & -z & -q_3
\end{pmatrix}  
\quad
\phi_1 = \begin{pmatrix} xz & yz & q_3 & q_4 \end{pmatrix}.
\end{equation}
Indeed, the columns of $\phi_2$ are clearly syzygies of $I$.  We note that $(q_3, q_4) \nsubseteq (z)$ since $\hgt(q_3, q_4) = 2$.  Since the first three columns are easily seen to be independent linear syzygies and the fourth column is not contained in the $S$-span of the first three columns by the preceding observation, it follows that columns of $\phi_2$ are part of a minimal set of generators for $\syz_1^S(I)$, and so, our assumption on Betti table of $R$ forces these syzygies to be all of the minimal generators for $\syz_1^S(I)$.  A similar argument shows that $\phi_3$ is the unique syzygy on these generators up to scalar multiple.  With the maps as above, it is clear that $\hgt (z, q_3, q_4, \Delta) = \hgt \a_3 \geq 3$ by the preceding proposition.  On the other hand, $(z, q_3, q_4, \Delta) = (z, I_2(M))$ where $M$ is the submatrix consisting of the first two rows and first three columns of $\phi_2$, and so, we have $\hgt (z, q_3, q_4) \leq 1 + \hgt I_2(M) \leq 3$ by \cite[2.1]{Bruns:Vetter} so that $\hgt (z, q_3, q_4) = 3$.

Conversely, when $I$ has the preceding form, we show that $R$ has the desired Betti table by explicitly constructing the minimal free resolution of $R$.  Consider the complex $F_\bullet$ of free graded $S$-modules as in \eqref{case:2:minimal:free:resolution} with differentials as in \eqref{case:2:differentials}.  We use the Buchsbaum-Eisenbud acyclicity criterion  \cite{what:makes:a:complex:exact?}  to show that \eqref{case:2:minimal:free:resolution} is the minimal free resolution of $R$.  We must check that $\hgt I_{r_i}(\phi_i) \geq i$ for each $i \geq 1$, where $r_i = \sum_{j \geq i} (-1)^{j-i}\rank F_j$.  In our case, it is clear that $\hgt I_1(\phi_1) = \hgt I \geq 1$ and that $I_1(\phi_3) = (z, q_3, q_4, \Delta)$ has height three by assumption.  Finally, one easily checks by computing minors that $(q_3, q_4)^2 \subseteq I_3(\phi_2)$, and thus, by assumption, $\hgt I_3(\phi_2) \geq 2$.
\end{proof}

\begin{cor} \label{1:linear:syzygy:case:is:LG-quadratic}
If $R = S/I$ is as in Theorem \ref{1:linear:syzygy:case}, then $R$ is LG-quadratic.
\end{cor}

\begin{proof}
We can consider the generic ideal of the form described by the theorem where the linear forms are all variables.  By the theorem and Corollary \ref{Betti:table:preserved:implies:quotient:by:regular:sequence}, any particular ideal $I$ of the above form must be obtained by specializing the generic case by a regular sequence of linear forms.  Thus, it suffices to prove that $I$ is G-quadratic in the generic case.  

In that case, we claim that the generators of $I$ are a Gr\"obner basis for any degree lexicographic order satisfying $a_3 > b_3 > b_4 > a_4 > x > y > z$.  Setting $q_1 = xz$ and $q_2 = yz$, we apply Buchberger's Criterion (see \cite{Herzog:Ene}) by checking that the S-pairs $S(q_i, q_j)$ reduce to zero upon division by $q_1, \dots, q_4$ for all $i < j$.  This is clear for $(i, j) = (1, 2)$, $(1, 4)$, $(2, 3)$, and $(3, 4)$ since either the S-pair is zero or the initial monomials of $q_i$ and $q_j$ are relatively prime in those cases.  In the remaining cases, we see that $S(q_1, q_3) = -b_3yz$ and $S(q_2, q_4) = -a_3xz$ also clearly reduce to zero so that the generators of $I$ form a Gr\"obner basis as claimed.
\end{proof}

We now turn our attention to the last remaining Betti table from Theorem \ref{Koszul:4:quadric:height:2:Betti:tables}.  This is the first instance where the Betti table does not determine whether the ring $R$ is Koszul.  

\begin{thm}\label{2:linear:syzygies:case}
The ring $R = S/I$ has Betti table
\vspace{1 ex}
\begin{center}
\textnormal{
\begin{tabular}{c|cccccccc}
  & 0 & 1 & 2 & 3 & 4 \\ 
\hline 
0 & 1 & -- & --  & -- & -- \\ 
1 & -- & 4 & 2  & --  & -- \\
2 & -- & -- & 4 & 4 & 1
\end{tabular}
}
\end{center}
\vspace{1 ex}
if and only if there are linear forms $x$, $y$, $a_i$, and $b_i$ with $\hgt(x, y) = 2$ such that $I$ has one of the following forms:
\begin{enumerate}[label = \textnormal{(\alph*)}]
\item $(x^2, b_3x, a_3x + b_3y, a_4x + b_4y)$, where $\hgt(x, b_3, b_4) = 3$ and $\hgt(x, y, a_3, b_3) = 4$;
\item $(xy, a_2x, b_3y, a_4x+b_4y)$, where $\hgt(x, b_3, b_4) = \hgt (y, a_2, a_4) = 3$ and $\hgt (x, y, a_2, b_3) = 4$;
\item $(b_3x, b_4x, a_3x + b_3y, a_4x + b_4y)$, where $\hgt(x, b_3, b_4) = 3$, $\hgt(a_3, a_4, b_3, b_4) = 4$, and $\hgt (a_3x + b_3y, a_4x + b_4y) = 2$;
\item $(a_1x, a_2x, b_3y, b_4y)$, where $(a_1x,a_2x)$ and $(b_3y, b_4y)$ are transversal ideals with $\hgt(a_1, a_2) = \hgt(b_3, b_4) = 2$.
\end{enumerate}
\end{thm}

To simplify the exposition of the proof of the theorem, we first state some preparatory lemmas.

\begin{rmk}
If $I_1, I_2 \subseteq S$ are transversal ideals and $z$ is a linear form which is a nonzerodivisor modulo $I = I_1 + I_2$, then the images of $I_1$ and $I_2$ in $S' = S/zS$ are still transversal.  Indeed, if $I'_j = (I_j, z)$ and $I' = (I, z)$, then applying the Snake Lemma to multiplication by $z$ on each term of the exact sequence 
\[ 0 \to S/(I_1 \cap I_2) \to S/I_1 \oplus S/I_2 \to S/I \to 0 \] 
yields an exact sequence $0 \to S/(I_1 \cap I_2, z) \to S/I'_1 \oplus S/I'_2 \to S/I' \to 0$ so that $I'_1 \cap I'_2 = (I_1 \cap I_2, z)$.  Since $I_1$ and $I_2$ are transversal, we see that $I'_1 \cap I'_2 = (I_1 \cap I_2, z) = (I_1I_2, z) = I'_1I'_2$ so that $I_1S' \cap I_2S' = (I'_1 \cap I'_2)S' = I'_1I'_2S' = (I_1S')(I_2S')$ as claimed.
\end{rmk}

\begin{lemma} \label{transversal:case}
Suppose that $I = (a_1x, a_2x, b_3y, b_4y) \subseteq S$ for some linear forms $x$, $y$, $a_i$, and $b_i$.  Then the Betti table of $R = S/I$ is
\vspace{1 ex}
\begin{center}
\textnormal{
\begin{tabular}{c|cccccccc}
  & 0 & 1 & 2 & 3 & 4 \\ 
\hline 
0 & 1 & -- & --  & -- & -- \\ 
1 & -- & 4 & 2  & --  & -- \\
2 & -- & -- & 4 & 4 & 1
\end{tabular}
}
\end{center}
\vspace{1 ex}
if and only if $(a_1x, a_2x)$ and $(b_3y, b_4y)$ are transversal ideals with $\hgt(a_1, a_2) = \hgt(b_3, b_4) = \hgt(x, y) = 2$.
\end{lemma}

\begin{proof}
Let $I_1=(a_1x,a_2x)$ and $I_2=(b_3y,b_4y)$.  If $I_1$ and $I_2$ are transversal and $\hgt(a_1, a_2) = \hgt(b_3, b_4) = 2$, then the minimal free resolution of $S/I_j$ is of the form $0 \longto S(-3) \longto S(-2)^2 \longto S$, and the tensor product of the resolutions of $S/I_1$ and $S/JI_2$ is a resolution of $R$ so that $R$ has the given Betti table. 

Conversely, suppose that $R$ has the given Betti table.  Then it is easily seen that $\hgt I = 2$, and $I$ is minimally generated by 4 quadrics, from which it follows that $\hgt (a_1, a_2) = \hgt (b_3, b_4) = \hgt (x, y) = 2$.  It remains to see that $I_1$ and $I_2$ are transversal.  This is clear in the generic case where the linear forms are all variables.  By Corollary \ref{Betti:table:preserved:implies:quotient:by:regular:sequence}, $I$ must be obtained by specializing the generic case by a regular sequence of linear forms, and so, $I_1$ and $I_2$ are transversal by the preceding remark.
\end{proof}

\begin{lemma} \label{non-Koszul:2:linear:syzygies:case}
If $I$ has one of the forms \textnormal{(a)--(c)} in Theorem \ref{2:linear:syzygies:case}, then $R = S/I$ has Betti table
\begin{center}
\textnormal{
\begin{tabular}{c|cccccccc}
  & 0 & 1 & 2 & 3 & 4 \\ 
\hline 
0 & 1 & -- & --  & -- & -- \\ 
1 & -- & 4 & 2  & -- & -- \\
2 & -- & -- & 4 & 4 & 1 
\end{tabular}
}
\end{center}
and is not Koszul.
\end{lemma}

\begin{proof}
We explicitly construct the minimal free resolution of $R$ for each of the forms (a)--(c) described in Theorem \ref{2:linear:syzygies:case}.  In each case, we let $q_i = a_ix + b_iy$ for $i = 3, 4$, and we let $J$ be the ideal generated by the first three generators of $I$.

{\sc Case} (a):  Since $\hgt(x, b_3, b_4) = 3$ and $\hgt(x, y, a_3, b_3) = 4$, we see that $b_3x,q_3$ is a regular sequence so that
\[ J = I_2 \begin{pmatrix} b_3 & -x & 0 \\ a_3 & y & -x \end{pmatrix}, \]
is Cohen--Macaulay of height 2.  Hence, $S/J$ has a Hilbert-Burch resolution \cite[1.4.17]{Bruns:Herzog}, from which it is easily computed that $e(S/J) = 3$. As $J$ is unmixed, its associated primes are minimal. Any prime ideal $P$ containing $J$ must contain $(x, b_3)$ and $(x, y)$, and so, these must be the only associated primes of $J$.  Since it is easily seen that $JS_{(x, y)} = (x, y)S_{(x, y)}$, we know that $J = (x, y) \cap L$ for some $(x, b_3)$-primary ideal $L$, which has multiplicity $e(S/L) = 2$ by the associativity formula.  It follows from the height conditions that $q_4 \notin (x, b_3)$ so that
\[ (J : q_4) = ((x, y) : q_4) \cap (L : q_4) = (L : q_4) = L. \]
By \cite[10]{Engheta}, we know that $(x, b_3)^2 + (q_3)$ is also an $(x, b_3)$-primary ideal of multiplicity 2, which is contained in $(J : q_4)$ since $b_3^2q_4 = \Delta b_3x + (b_3b_4)q_3$ where $\Delta = a_3b_4 - b_3a_4$. Thus, $(J : q_4) = L = (x, b_3)^2 + (q_3)$ by \cite[8]{Engheta}.  

Since the minimal free resolutions of $S/(J : q_4)$ is well-known (see \cite[10]{Engheta}), we can lift the map ${S/(J : q_4)(-2) \stackrel{q_4}{\to} S/J}$ to a chain map between their minimal free resolutions:
\[
\scalebox{0.9}{
\begin{tikzcd}[ampersand replacement = \&, column sep = 2.5 cm, row sep = 1.5 cm]
0 \rar \& 
S(-6) \rar{\left(\begin{smallmatrix} a_3 \\ y \\ -b_3 \\ x \end{smallmatrix}\right)} \& 
S(-5)^4 \rar{\left(\begin{smallmatrix} b_3 & 0 & a_3 & 0 \\ -x & b_3 & y & a_3 \\ 0 & -x & 0 & y \\ 0 & 0 & -x & - b_3 \end{smallmatrix}\right)} \dar[swap]{\left(\begin{smallmatrix} q_4 & -a_3b_4 & 0 & -a_3a_4 \\ 0 & -b_3b_4 & -q_4 & -b_3a_4 \end{smallmatrix}\right)} \& 
S(-4)^4 \rar{\left(\begin{smallmatrix} x^2 & b_3x & b_3^2 & q_3 \end{smallmatrix}\right)} \dar{\left(\begin{smallmatrix} q_4 & 0 & 0 & 0 \\ 0 & q_4 & -\Delta & 0 \\ 0 & 0 & b_3b_4 & q_4 \end{smallmatrix}\right)} \& 
S(-2) \dar{q_4}
\\
\& 
0 \rar \& 
S(-3)^2 \rar{\left(\begin{smallmatrix} b_3 & a_3 \\ -x & y \\ 0 & -x \end{smallmatrix}\right)} \& 
S(-2)^3 \rar{\left(\begin{smallmatrix} x^2 & b_3x & q_3 \end{smallmatrix}\right)} \& 
S
\end{tikzcd}
}
\]
Taking the mapping cone of this chain map yields the minimal free resolution of $R$.  In particular, we see that $\syz_1^S(I)$ is minimally generated by the columns of the matrix
\[ \begin{pmatrix} b_3 & a_3 & q_4 & 0 & 0 & 0 \\ -x & y & 0 & q_4 & -\Delta & 0 \\ 0 & -x & 0 & 0 & b_3b_4 & q_4 \\ 0 & 0 & -x^2 & -b_3x & -b_3^2 & -q_3 \end{pmatrix}. \]
If $R$ were Koszul, then $\syz_1^S(I)$ would be minimally generated by linear syzygies and Koszul syzygies on the minimal generators of $I$ by \cite[2.8]{Koszul:ACI's}, but in this case, the linear and Koszul syzygies are spanned by all but the fifth column of the above matrix so that $R$ cannot be Koszul.

{\sc Case} (b):  Since $\hgt(x, y, a_2, b_3) = 4$, we see that $J = (x, y) \cap (x, b_3) \cap (y, a_2)$ so that $J$ is a Cohen--Macaulay ideal of $\hgt J = 2$.  Clearly $q_4 \in (x, y)$. On the other hand, since $\hgt (x, b_3, b_4) = \hgt (x, y, b_3) = \hgt(x, a_2, a_4) = \hgt (x, y, a_2) = 3$, we see that $q_4 \notin (x, b_3) \cup  (y, a_2)$. Thus $(J : q_4) = (x, b_3) \cap (y, a_2)$.

Now, the ring $S/J$ has a Hilbert-Burch resolution, and the resolution of ${S/(J : q_4)}$ is easily constructed using the fact that $\hgt(x, y, a_2, b_3) = 4$ (see for example \cite[3.3]{Koszul:algebras:defined:by:4:quadrics}). We lift the map ${S/(J : q_4)(-2) \stackrel{q_4}{\to} S/J}$ to a chain map between their minimal free resolutions as follows
\[
\scalebox{0.9}{
\begin{tikzcd}[ampersand replacement = \&, column sep = 2.5 cm, row sep = 1.5 cm]
0 \rar \& 
S(-6) \rar{\left(\begin{smallmatrix} b_3 \\ -a_2 \\ -y \\ x \end{smallmatrix}\right)} \& 
S(-5)^4 \rar{\left(\begin{smallmatrix} -a_2 & -b_3 & 0 & 0 \\ y & 0 & b_3 & 0 \\ 0 & x & 0 & a_2 \\ 0 & 0 & -x & - y \end{smallmatrix}\right)} \dar[swap]{\left(\begin{smallmatrix} q_4 & 0 & b_3b_4 & -a_4b_3 \\ 0 & q_4 & -a_2b_4 & a_2a_4 \end{smallmatrix}\right)} \& 
S(-4)^4 \rar{\left(\begin{smallmatrix} xy & a_2x & b_3y & a_2b_3 \end{smallmatrix}\right)} \dar{\left(\begin{smallmatrix} q_4 & 0 & 0 & 0 \\ 0 & q_4 & 0 & a_4b_3 \\ 0 & 0 & q_4 & a_2b_4 \end{smallmatrix}\right)} \& 
S(-2) \dar{q_4}
\\
\& 
0 \rar \& 
S(-3)^2 \rar{\left(\begin{smallmatrix} -a_2 & -b_3 \\ y & 0 \\ 0 & x \end{smallmatrix}\right)} \& 
S(-2)^3 \rar{\left(\begin{smallmatrix} xy & a_2x & b_3y \end{smallmatrix}\right)} \& 
S
\end{tikzcd}
}
\]
and taking the mapping cone of this chain map yields the minimal free resolution of $R$.  In particular,  $\syz_1^S(I)$ is minimally generated by the columns of the matrix
\[ \begin{pmatrix} -a_2 & -b_3 & q_4 & 0 & 0 & 0 \\ y & 0 & 0 & q_4 & 0 & a_4b_3 \\ 0 & x & 0 & 0 & q_4 & a_2b_4 \\ 0 & 0 & -xy & -a_2x & -b_3y & -a_2b_3 \end{pmatrix}. \]
As above, the linear and Koszul syzygies are spanned by only the first five columns of the above matrix so that $R$ cannot be Koszul by \cite[2.8]{Koszul:ACI's}.

{\sc Case} (c):  After a suitable change of generators for $I$, we may assume that $\hgt(x, a_3, b_3, b_4) = 4$.  This is clear if $\hgt (x, a_3, a_4, b_3, b_4) = 5$.  Otherwise, we must have $\hgt(x, a_3, a_4, b_3, b_4) = 4$ since $\hgt (a_3, a_4, b_3, b_4) = 4$.  In that case, since $\hgt(x, b_3, b_4) = 3$, we may assume after possibly interchanging $q_3$ and $q_4$ that $a_4 \in (x, a_3, b_3, b_4)$ and $\hgt(x, a_3, b_3, b_4) = 4$ as claimed.  

Since $\hgt(x, y) = 2$ and $\hgt(x, a_3, b_3, b_4) = \hgt(a_3, a_4, b_3, b_4) = 4$, it is clear that $q_3$ is a nonzerodivisor modulo $(b_3x, b_4x) = (x) \cap (b_3, b_4)$.  Hence, the minimal free resolution of $S/J$ has the form
\[ 
0 \longto S(-5) \longto S(-4)^2 \oplus S(-3) \longto S(-2)^3 \longto S \longto S/J \longto 0.
\]

Next, we claim that $(J: q_4) = (b_3, b_4x, q_3) = (b_3, b_4x, a_3x)$.  Indeed, it is clear that the colon contains the latter ideal.  If $q_4f \in (b_3x, b_4x, q_3)$, then $b_4yf \equiv b_3yg \pmod{x}$ for some $g \in S$ so that $f \in (b_3, x)$ as $\hgt(x, y) = 2$ and $\hgt(x, b_3, b_4) = 3$.  Similarly, we see that $a_4xf \equiv 0 \pmod{(b_3, b_4, a_3)}$ so that $f \in (b_3, b_4, a_3)$ as $\hgt(a_3, a_4, b_3, b_4) = \hgt (x, a_3, b_3, b_4) = 4$.  And so, it follows that $(J : q_4) \subseteq (b_3, x) \cap (b_3, b_4, a_3) = (b_3, b_4x, a_3x)$.  As $b_3$ is a nonzerodivisor modulo $(b_4x, a_3x)$, we see that the minimal free resolution of $S/(J: q_4)$ has the form
\[
0 \longto S(-4) \longto S(-3)^3 \longto S(-2)^2 \oplus S(-1) \longto  S\longto S/(J : q_4) \longto 0.
\]
We obtain a free resolution of $R$ by taking the mapping cone of the chain map lifting multiplication by $q_4$ in the short exact sequence
\[
0 \longto S/(J:q_4)(-2) \stackrel{q_4}{\longto} S/J \longto R \longto 0.
\]
The degrees of the free modules in the resolutions of $S/J$ and $S/(J : q_4)$ force this lift to have entries of positive degree.   Hence, the resulting free resolution is minimal, and $R$ has the desired Betti table.  However, unlike cases (a) and (b) above, in this case the first syzygies of $I$ are minimally generated by linear syzygies and Koszul syzygies on the minimal generators of $I$, making it more difficult to prove that $R$ is not Koszul. We defer the proof of this fact to Appendix \ref{bad:non-Koszul:algebra} at the end of this paper.
\end{proof}

Our approach to the remaining parts of the proof of Theorem \ref{2:linear:syzygies:case} is based around an analysis of the matrix of linear first syzygies on the defining ideal $I$.

\begin{lemma}
Suppose $R = S/I$ has Betti table:
\begin{center}
\textnormal{
\begin{tabular}{c|cccccccc}
  & 0 & 1 & 2 & 3 & 4 \\ 
\hline 
0 & 1 & -- & --  & -- & -- \\ 
1 & -- & 4 & 2  & -- & -- \\
2 & -- & -- & 4 & 4 & 1 
\end{tabular}
}
\end{center}
Then:
\begin{enumerate}[label = \textnormal{(\alph*)}]
\item There is a unique height two linear prime $P$ containing $I$, and $(I : P)$ is nondegenerate.
\item The matrix of linear syzygies of $I$ has a generalized zero.
\item If the coordinates of any linear syzygy $\ell = (\ell_1, \ell_2, \ell_3, \ell_4)$ of $I$ generate an order ideal of height 3, then the matrix of linear syzygies of $I$ has a generalized row which is zero.
\end{enumerate}
\end{lemma}

\begin{proof}
Let $q_1, \dots, q_4$ denote the minimal quadratic generators of $I$.

(a)  It is easily computed that $e(R) = 1$ and $\hgt I = 2$.  Hence, there is a unique height two prime $P$ containing $I$, which is necessarily a linear prime.  If $(I: P)$ is degenerate, then we can write $I = (xz, yz, a_3x+b_3y, a_4x+b_4y)$ for some linear forms $x, y, z, a_i, b_i$ with $P = (x, y)$.  However, in that case, the columns of the matrix
\[
\begin{pmatrix}
y & a_3 & a_4\\
-x & b_3 & b_4 \\
0 & -z & 0 \\
0 & 0 & -z
\end{pmatrix}
\]
are 3 linearly independent linear syzygies on $I$, contradicting the assumption on the Betti table of $R$.

(b) Suppose the $4 \times 2$ matrix $M$ of linear syzygies of $I$ is 1-generic.  Then the vector $q = (q_1, q_2, q_3, q_4) \in S(-1)^4$ is a syzygy of the module $\coker M^\T$. Since $M$ is $1$-generic, then $I_2(M)$ is a prime ideal of height 3 generated by quadrics \cite[6.4]{geometry:of:syzygies}, and the minimal free resolution of $\coker M^\T$ is given by a Buchsbaum-Rim complex \cite[A2.13]{Eisenbud}.  It follows from the description of the second differential of the Buchsbaum-Rim complex that $I \subseteq I_2(M)$.  However, $I$ contains a reducible quadric by {\cite[4.7]{projective:dimension:of:height:2:ideals:of:quadrics}} and \cite[4.3]{Koszul:algebras:defined:by:4:quadrics}. Since $I_2(M)$ is prime, then $I_2(M)$ contains a linear form, which is impossible.

(c)  After a suitable change of generators for $I$, we may assume that $\ell_4 = 0$ and $\hgt(\ell_1, \ell_2, \ell_3) = 3$.  We can then write 
\[
\begin{pmatrix} q_1 \\ q_2 \\ q_3 \end{pmatrix} =
\begin{pmatrix} 
\ell_2 & \ell_3 & 0 \\
-\ell_1 & 0 & \ell_3 \\
0 & -\ell_1 & -\ell_2
\end{pmatrix}
\begin{pmatrix} u_1 \\ u_2 \\ u_3 \end{pmatrix} =
\begin{pmatrix} \ell_2u_1 + \ell_3u_2 \\ -\ell_1u_1 + \ell_3u_3 \\ -\ell_1u_2 -\ell_2u_3 \end{pmatrix}
\]
for some linear forms $u_i$ by viewing the $q_j$'s as a syzygy on the $\ell_j$'s.  But then $I \supseteq (q_1, q_2, q_3) = I_2(M)$, where
\[ M = \begin{pmatrix} \ell_1 & \ell_2 & \ell_3 \\ u_3 &  -u_2 & u_1 \end{pmatrix}. \] 
If $I_2(M) = (q_1, q_2, q_3)$ has height one, then $R$ does not have the assumed Betti table by \cite[3.3]{Koszul:algebras:defined:by:4:quadrics}.  Hence, $I_2(M)$ has height 2 so that the rows of $M$ represent the two independent linear syzygies of $I$ by \cite[1.4.17]{Bruns:Herzog}.
\end{proof}

\begin{proof}[Proof of Theorem \ref{2:linear:syzygies:case}]
By Lemmas \ref{transversal:case} and \ref{non-Koszul:2:linear:syzygies:case}, it remains only to prove that, if $R$ has the given Betti table, then $I$ has one of the forms (a)--(d).  As in the preceding lemma, we see that $e(R) = 1$ and $\hgt I = 2$, and we let $P$ denote the unique height two linear prime containing $I$.   We consider two cases based on whether the matrix of linear syzygies of $I$ has a generalized row that is zero.  This is precisely the distinction between the possible forms (a)--(b) versus (c)--(d) listed in the theorem.

{\sc Case 1:} Suppose the matrix of linear syzygies of $I$ has a generalized row which is zero.  We can then choose generators $q_1, \dots, q_4$ for $I$ such that $J = (q_1, q_2, q_3)$ has exactly two independent linear syzygies.  It then follows from \cite[4.2]{Koszul:ACI's} that $J = I_2(M)$ for some $3 \times 2$ matrix $M$ of linear forms.  The ideal $J$ is not prime because it is strictly contained in the height two linear prime $P$, so the matrix $M$ is not 1-generic by \cite[6.4]{geometry:of:syzygies}, and thus $M$ has at least one generalized zero. We may then assume 
\[ 
M = \begin{pmatrix}
b & d \\
-a & c \\
0 & -x
\end{pmatrix}
\]
for some linear forms $a, b, c, d, x$.  In particular, we see that $\hgt(a, b) = 2$ and $x \neq 0$ as $J = (ax, bx, ad + bc)$ is minimally generated by 3 quadrics. 

Consider the linear prime $(P,a,b)$, which must have height between 2 and 4. We claim that $\hgt (P,a,b)=3$.  If we have $\hgt(P, a, b) = 2$, then $P = (a, b)$ so that $x \in (J : P) \subseteq (I : P)$, contradicting that the latter ideal is nondegenerate.  In particular, this would be the case if $x \notin P$, so we must have $x \in P$.  Let $y$ be any linear form such that $P = (x, y)$.  If $\hgt(P, a, b) = 4$, then $ad + bc \in J \subseteq P$ implies that 
\[ \begin{pmatrix} d \\ c \end{pmatrix} \equiv \lambda \begin{pmatrix} b \\ -a \end{pmatrix} \pmod{P} \]
for some some $\lambda \in \kk$.  Hence, after a suitable change of generators for $J$, we may assume $J = (ax, bx, zy)$ for some nonzero linear form $z \in (a, b)$.  But then $z \in (I : P)$, which is a contradiction.  Therefore, $\hgt(P, a, b) = 3$, and without loss of generality, we may assume that $\hgt(P, b) = 3$ and that $a \in P$.  As $ad + bc \in P$, this implies that $c \in P$.  Note also that $(a, c) \nsubseteq (x)$ or else we would have $\hgt J = 1$.  

If $a \notin (x)$, then $P = (x, a)$, and after a suitable change of generators for $J$ we may assume that $c = 0$ so that $J = (ax, bx, ad)$.  Moreover, we note that $\hgt(x, a, b, d) = 4$.  Indeed, if this were not the case, we would have $d \in (x, a, b) = (P, b)$ so that, after a change of generators for $J$, we may assume that $d \in (a, b)$.  But then $d \in (J : P) \subseteq (I: P)$, which is a contradiction.  Relabeling $x, a, b, d$ as $x, y, a_2, b_3$ respectively and writing the last generator of $I$ as $a_4x + b_4y$, we see that $I$ has the form (b) above.  It remains only to check that $\hgt(x, b_3, b_4) = \hgt (y, a_2, a_4) = 3$.  If $\hgt(x, b_3, b_4) < 3$, then $b_4 \in (x, b_3)$ so that after a suitable change of generators we may assume that $b_4 = 0$.  But then $R$ does not have the assumed Betti table by \cite[3.3]{Koszul:algebras:defined:by:4:quadrics}.  Hence, we must have $\hgt(x, b_3, b_4) = 3$, and similarly, we see that $\hgt(y, a_2, a_4) = 3$. 

On the other hand, if $a \in (x)$, then we must have $P = (x, c)$ so that $J = (x^2, bx, dx + bc)$.  Similarly, we must have $\hgt( x, b, c, d) = 4$.  Again, if this were not the case, we would have $d \in (x, b, c) = (P, b)$ so that, after a change of generators for $J$, we may assume that $d \in (c)$ and, hence, we have $b + \lambda x \in (J: P)$ for some $\lambda \in \kk$, which is a contradiction.  Relabeling $x, b, c, d$ as $x, b_3, y, a_3$ respectively and writing the last generator of $I$ as $a_4x + b_4y$, we see that $I$ has the form (a), where $\hgt(x, b_3, b_4) = 3$ by arguing as above. 

{\sc Case 2:}  Suppose $I$ does not have a generalized row that is zero.  By the preceding lemma, the matrix of linear syzygies of $I$ has a generalized zero.  Hence, $I$ has a nonzero linear syzygy whose order ideal has height at most 3, and so, its order ideal is easily seen to have height two by our assumption and the preceding lemma.  Consequently, after a suitable change of generators for $I$, the matrix of linear syzygies of $I$ has the form
\[
M = \begin{pmatrix}
d & s \\ -c & u \\ 0 & v \\ 0 & -w
\end{pmatrix}
\]
for some linear forms with $\hgt(c, d) = 2$.  If $q_1, \dots, q_4$ denote the corresponding generators of $I$, then $dq_1 - cq_2 = 0$ implies $q_1 = cx$ and $q_2 = dx$ for some nonzero linear form $x$.  We see that $x \in P$, since otherwise we would have $P=(c, d)$ so that $x \in (I : P)$, contrary to the preceding lemma.  Note also that $\hgt(v, w) = 2$ by our assumption.  

If $\hgt (s, u, v, w) = 2$, then $s, u \in (v, w)$.  Hence, after a suitable change of generators for $I$, we may assume that $s = u = 0$.  In that case, arguing as above shows that $q_3 = wy$ and $q_4 = vy$ for some linear form $y \in P$.  We note that $P = (x, y)$ has height two since $\hgt I = 2$.  Thus, $I$ has the form (d) by Lemma \ref{transversal:case}, and we may assume that $\hgt(s, u, v, w) = 4$ for the remainder of the proof.

Let $y$ be any linear form such that $P = (x, y)$, and write $q_i = a_ix + b_iy$ for some linear forms $a_i, b_i$ for $i = 3, 4$.  We claim that $\hgt (x, v, w) = 3$.  If not, then after a suitable change of generators for $I$, we may assume that $x = w$ and $\hgt(x, v) = 2$.  Then, considering the syzygy represented by the second column of $M$ modulo $x$, we have $vb_3y \equiv 0 \pmod{x}$ so that $b_3 \equiv 0 \pmod{x}$.  Hence, $(q_1, q_2, q_3) \subseteq (x)$ so that $R$ cannot have the assumed Betti table by \cite[3.3]{Koszul:algebras:defined:by:4:quadrics}.  

Since $\hgt(x, v, w) = 3$, considering the same relation modulo $(v, w)$ yields $scx + udx \equiv 0 \pmod{(v, w)}$ so that $sc + ud \equiv 0 \pmod{(v, w)}$.  Since $\hgt (s, u, v, w) = 4$, we can write
\[
\begin{pmatrix} c \\ d \end{pmatrix} \equiv
\lambda\begin{pmatrix} -u \\ s \end{pmatrix}  \pmod{(v, w)}
\]
for some $\lambda \in \kk$.  If $\lambda \neq 0$, then after subtracting $\lambda^{-1}$ times the first column of $M$ from the second, we may assume that $\hgt(s, u, v, w) = 2$.  And so, as above, we see that $I$ has form (d).   Hence, we may assume that $\lambda = 0$.  

In that case, we have $c, d \in (v, w)$ so that $(c, d) = (v, w)$ as $\hgt (c, d) = 2$.  Then, after a suitable change of generators for $I$, we may assume that $v = d$ and $w = -c$.  Then, considering the syzygy represented by the second column of $M$ modulo $x$, we have $db_3y - cb_4y \equiv 0 \pmod{x}$ so that $db_3 - cb_4 \equiv 0 \pmod{x}$ as $\hgt(x, y) = 2$ and
\[
\begin{pmatrix} -b_4 \\ b_3 \end{pmatrix} \equiv
\alpha\begin{pmatrix} -d \\ c \end{pmatrix}  \pmod{x}
\]
as $\hgt(x, c, d) = 3$.  If $\alpha = 0$, then $I \subseteq (x)$, contradicting that $\hgt I = 2$. Hence, $\alpha \neq 0$, and after replacing $a_3, a_4$ with different linear forms, we may assume that $b_3 = c$ and $b_4 = d$.  Then the syzygy represented by the second column of $M$ is
\[ 0 = sb_3x + ub_4x + b_4q_3 -b_3q_4 = (s - a_4)b_3x + (u + a_3)b_4x \]
so that $0 = (s - a_4)b_3 + (u + a_3)b_4$.  Since $\hgt(b_3, b_4) = 2$, we finally see that
\[
\begin{pmatrix} a_4 \\ -a_3 \end{pmatrix} =
\begin{pmatrix} s \\ u \end{pmatrix} + \beta\begin{pmatrix} b_4 \\ -b_3 \end{pmatrix}
\]
for some $\beta \in \kk$ so that $(a_3, a_4, b_3, b_4) = (s, u, b_3, b_4)$ has height four.  If $\hgt (q_3, q_4) = 1$, then arguing as above shows that $I$ has form (d), and otherwise, $\hgt(q_3, q_4) = 2$ so that $I$ has form (c) of the theorem.
\end{proof}

\begin{proof}[Proof of Theorem \ref{LG-quadratic}] 
Suppose $R = S/I$ is a Koszul algebra defined by $g \leq 4$ quadrics.  It is known that $R$ is LG-quadratic when $\hgt I = g$ \cite[1.2.5]{Caviglia:PhD:thesis} and when $\hgt I = g - 1$ \cite[3.4]{Koszul:ACI's}.  

If $\hgt I = 1$, then $I = zJ$ for some linear form $z$ and linear complete intersection $J$ so that the Betti numbers of $R$ depend only on the number $g$ of minimal generators of $I$.  By Corollary \ref{Betti:table:preserved:implies:quotient:by:regular:sequence}, any particular height one ideal must be obtained by specializing the generic case where $\hgt(z, J) = g + 1$ by a regular sequence of linear forms, and in the generic case, $I$ is a monomial ideal after suitable change of coordinates and, thus, is G-quadratic.  

Hence, it only remains to see that $R$ is LG-quadratic when $g = 4$ and $\hgt I = 2$.  We consider the possible forms for $I$ described by Theorem \ref{structure:theorem}.  In cases (i) and (ii), $R$ is LG-quadratic by \cite[3.5]{Koszul:algebras:defined:by:4:quadrics}.  In cases (iii) and (iv), this follows from Corollary \ref{1:linear:syzygy:case:is:LG-quadratic} and Proposition \ref{transveral:height:1:ideals:are:LG-quadratic} respectively.
\end{proof}

\section{Applications to the Backelin--Roos Property}  \label{Backelin-Roos:property:section}

As an application of our structure theorem, we prove that all Koszul algebras defined by $g \leq 4$ quadrics, except the case where $I$ is the sum of two transversal height one ideals, have the Backelin-Roos property.  Recall that $R$ has the \emph{Backelin--Roos property} if there exists a surjective Golod homomorphism $\phi: Q \to R$ of standard graded $\kk$-algebras where $Q$ is a complete intersection; the ring $R$ is \emph{absolutely Koszul} if every finitely generated graded $R$-module has finite linearity defect (see \cite{absolutely:Koszul:algebras} for further details). When $R$ is Koszul, it is well-known that
\begin{center}
\begin{tikzcd}
\text{Backelin--Roos property} \rar[Rightarrow] & \text{Absolutely Koszul}
\end{tikzcd}
\end{center}
and it is an open question whether the two notions are equivalent for Koszul algebras up to field extensions \cite[p.~354]{absolutely:Koszul:algebras}.  Using our structure theorem, we give an affirmative answer to this question for Koszul algebras defined by $g \leq 4$ quadrics in characteristic zero.

In the following theorem, we make use of certain properties of algebra retracts.  Recall that a map of standard graded $\kk$-algebras $\phi: R \to R'$ is an \emph{algebra retract} if there exists a $\kk$-algebra map $\sigma: R' \to R$ such that $\phi \circ \sigma = \Id_{R'}$.

\begin{thm}
Let $R = S/I$ be a Koszul algebra defined by $g \leq 4$ quadrics over an algebraically closed field $\kk$. Then:
\begin{enumerate}[label = \textnormal{(\alph*)}] 
\item $R$ has the Backelin-Roos property except possibly when $I$ is the sum of two transversal height one ideals each generated by two quadrics.
\item If $I$ is the sum of two transversal height one ideals each generated by two quadrics and $\ch(\kk) = 0$, then $R$ is not even absolutely Koszul. 
\end{enumerate}
\end{thm}

\begin{proof}
Clearly, $R$ has the Backelin-Roos property if it is a quadratic complete intersection, in which case $\hgt I = g$.  By \cite[3.1]{absolutely:Koszul:algebras}, $R$ also has the Backelin-Roos property if $\reg_S R \leq 1$.  In particular, this is the case when $\hgt I = 1$. It also follows that factoring out a regular sequence of linear forms or quadrics preserves the Backelin-Roos property by considering one nonzerodivisor at a time.  If $\hgt I = {g - 1}$, then by \cite{Koszul:ACI's} either $I = I_2(M) + (q_4, \dots, q_g)$ for some $3 \times 2$ matrix $M$ of linear forms with $\hgt I_2(M) = 2$ and quadrics $q_4, \dots, q_g$ that form a regular sequence modulo $I_2(M)$, or $I = (a_1x, a_2x, q_3, \dots, q_g)$ for some linear forms $x, a_1, a_2$ and quadrics $q_3, \dots, q_g$ that form a regular sequence modulo $(a_1x, a_2x)$.  Since $\reg_S S/I_2(M) = \reg_S S/(a_1x, a_2x) = 1$, these quotients have Backelin--Roos property, and thus, so does $R$.  Hence, the proof of the theorem reduces to considering the case when $\hgt I = 2$ and $I$ is generated by $g=4$ quadrics.  

In that case, we consider the 4 possible forms for $I$ described by Theorem \ref{structure:theorem}. In cases (i) and (ii), it follows from \cite[3.3]{Koszul:algebras:defined:by:4:quadrics} that either $\reg_S R = 1$ or $R$ can be obtained by factoring out a quadratic nonzerodivisor from a ring of regularity one. Thus, as above, $R$ has the Backelin-Roos property.

{\sc Case} (iv): In this case, $I = (ax, bx, cy, dy)$ is a sum of two transversal height one ideals each generated by two quadrics.  Since the absolutely Koszul property ascends from quotients of $R$ by a regular sequence of linear forms \cite[2.4]{absolutely:Koszul:algebras}, it suffices to show that $R$ is not absolutely Koszul in the generic case where $a, b, c, d, x, y$ are all variables of $S$.  

Under these circumstances, we claim that $R$ is a bad Koszul algebra in the sense of Roos \cite{GoodBad} and, therefore, is not absolutely Koszul \cite[1.8]{Koszul:modules}.  Being a good Koszul algebra descends to quotients rings under maps that are large homomorphisms in the sense of Levin \cite[2.5]{GoodBad}, of which two important cases are algebra retracts and quotients by a regular sequence of linear forms \cite[2.2, 2.3]{Levin}.  Thus, by an algebra retract it first suffices to show that $R$ is a bad Koszul algebra when $S = \kk[a, b, c, d, x, y]$, and then by killing the regular sequence $x - a$, $y - d$, we may further assume that $R = \kk[a, b, c, d]/(a^2, ab, cd, d^2)$, which Roos shows is a bad Koszul algebra in characteristic zero \cite[2.4]{GoodBad}.
 
{\sc Case} (iii):  In this case, we have $I = (xz, yz, a_3x+b_3y, a_4x + b_4y)$ for some linear forms  $x$, $y$, $z$, $a_i$, $b_i$ such that $\hgt(x, y) = 2$, $\hgt (a_3x+b_3y, a_4x+b_4y) = 2$, and $\hgt (z, a_3x+b_3y, a_4x+b_4y, a_3b_4-a_4b_3) = 3$. Set $q_i = a_ix+b_iy$ for $i = 3, 4$, $\Delta = a_3b_4 - a_4b_3$, and $C=(q_3,q_4)$. Then $C$ is a complete intersection, and we claim the natural surjection $\varphi: Q = S/C \to R$ is Golod. Since $R$ is Koszul by Corollary \ref{1:linear:syzygy:case:is:LG-quadratic} and since $Q$ is Koszul, then by \cite[3.1]{absolutely:Koszul:algebras} it suffices to show that $I/C$ has a 2-linear resolution over $Q$. 

We first consider the case where $\hgt (z, q_3, q_4) = 3$.  In that case, it suffices to show that $I/C = z(x, y)Q \iso (x,y)Q$ has a linear resolution as a $Q$-module. Write $S=S'[x,y]$ where $S'$ is a polynomial ring over $k$. Since $C\subseteq (x,y)$, the natural epimorphism $\pi : Q \to Q/(x,y)Q \iso S'$ is an algebra retract. Indeed, if $\Phi : S' \to Q$ denotes the composition $S'\hookrightarrow S \to Q$, it is easily seen that $\pi\circ \Phi = \Id_{S'}$. Since $Q$ is Koszul, it follows from \cite[2.3]{absolutely:Koszul:algebras} that $\ld_Q Q/(x,y)Q = 0$.  Thus, $(x,y)Q$ has a linear resolution as a $Q$-module so that $\varphi : Q \to R$ is Golod.

In general, consider the ideal $\tilde{I} = (wx, wy, q_3, q_4) \subseteq \tilde{S} = S[w]$, which has the same form as $I$ but where $\hgt(w, q_3, q_4) = 3$ since $(w, q_3, q_4)/w\tilde{S} \iso C$ has height 2.  By the preceding paragraph, we know that the surjection $\tilde{S}/C\tilde{S} \to \tilde{S}/\tilde{I}$ is Golod.  By arguing as in \cite[3.9]{absolutely:Koszul:algebras}, it suffices to show that $w - z$ is regular mod $\tilde{I}$ and mod $C\tilde{S}$ to prove that the induced surjection $S/C \to R$ is Golod so that $R$ has the Backelin-Roos property.  That $w - z$ is regular on $\tilde{S}/\tilde{I}$ follows from Corollary \ref{Betti:table:preserved:implies:quotient:by:regular:sequence} since it has the same Betti table over $\tilde{S}$ as the Betti table of $R$ over $S$ by Theorem \ref{1:linear:syzygy:case}.  On the other hand, any associated prime of $C\tilde{S}$ is of the form $P\tilde{S}$ for some associated prime $P$ of $C$ in $S$ \cite[23.2]{Matsumura}, and so, $w - z \notin P\tilde{S}$ since otherwise we would have $w \in (z, P)\tilde{S}$, which is clearly impossible.  Hence, $w -z$ is regular mod $C\tilde{S}$.
\end{proof}

Combining the above theorem with \cite[3.4]{absolutely:Koszul:algebras}, we have the following.

\begin{cor}
For a Koszul algebra $R = S/I$ defined by $g \leq 4$ quadrics over an algebraically closed field of characteristic zero, the following are equivalent:
\begin{enumerate}[label = \textnormal{(\alph*)}]
\item $R$ has the Backelin--Roos property.
\item $R$ is absolutely Koszul.
\item $R$ does not have the Betti table:
\begin{center}
\textnormal{
\begin{tabular}{c|cccccccc}
  & 0 & 1 & 2 & 3 & 4 \\ 
\hline 
0 & 1 & -- & --  & -- & -- \\ 
1 & -- & 4 & 2  & --  & -- \\
2 & -- & -- & 4 & 4 & 1
\end{tabular}
}
\end{center}
\item $I$ is not the sum of two transversal height one ideals each generated by two quadrics. 
\end{enumerate}
\end{cor}

\begin{rmk}
Roos' proof in \cite{GoodBad} that $\kk[a, b, c, d]/(a^2, ab, cd, d^2)$ is a bad Koszul algebra relies on $\kk$ having characteristic zero.  The results of this section would hold in arbitrary characteristic if we knew that this ring is a bad Koszul algebra without any restriction on $\kk$.
\end{rmk}

\begin{appendix}
\section{The Bad Non-Koszul Algebra}
\label{bad:non-Koszul:algebra}

In this section, we show that ideals of the form (c) from Theorem \ref{2:linear:syzygies:case} never define a Koszul algebra, tying up the remaining loose end in our structure theorem for Koszul algebras defined by four quadrics.  

Suppose that $R = S/I$ where $I = (b_3x, b_4x, a_3x + b_3y, a_4x + b_4y)$ for some linear forms $x, y, a_i, b_i$ satisfying $\hgt(x, b_3, b_4) = 3$, $\hgt(a_3, a_4, b_3, b_4) = 4$, and $\hgt (a_3x + b_3y, a_4x + b_4y) = 2$ as in the theorem.  In particular, we can consider the generic ideal of this form where the linear forms are all variables.  By Lemma \ref{non-Koszul:2:linear:syzygies:case} and Corollary \ref{Betti:table:preserved:implies:quotient:by:regular:sequence}, any particular ideal $I$ of the above form must be obtained by specializing the generic case by a regular sequence of linear forms, and since the Koszul property ascends and descends along quotients by regular sequences of linear forms by Proposition \ref{passing:Koszulness:to:and:from:quotients}, it suffices by first ascending to the generic case and then specializing again to prove that \emph{at least one} ideal $I$ of the above form is not Koszul.

\begin{notation}
For the remainder of this section, we let
\[ R = \kk[x, y, a, b]/(bx, xy, ax-by, x^2 -y^2), \]
which is easily seen to have the form (c) from Theorem \ref{2:linear:syzygies:case}.  We also let
\[ R' = \kk[x, y]/(x^2 -y^2, xy), \qquad M = \coker_{R'} A, \qquad A = \begin{pmatrix} x & 0 \\ -y & x \end{pmatrix}, \]
and we observe that $R \iso \Sym_{R'}(M)$.  In particular, $R$ is bigraded by setting $\deg x = \deg y = (1, 0)$ and $\deg a = \deg b = (0, 1)$.  Moreover, it is easily checked that $M$ has a periodic resolution over $R'$ similar to the cokernel of a matrix factorization over a hypersurface:
\[ \cdots \longto R'(-3)^2 \stackrel{A}{\longto} R'(-2)^2 \stackrel{B}{\longto} R'(-1)^2 \stackrel{A}{\longto} (R')^2 \longto M \longto 0, \qquad B = \begin{pmatrix} y & 0 \\ x & y \end{pmatrix}. \]
We will show that this particular ring $R$ is not Koszul.  Thus, not even a symmetric algebra of a module having a linear resolution over a quadratic complete intersection is necessarily Koszul.
\end{notation}

Unfortunately, even in this special setup, showing that $R$ is not Koszul seems to require some technical computation.  To simplify the computation as much as possible, we will exploit the fact that $R'$ is an algebra retract of $R$.

\begin{prop}[{\cite[1.4]{HHO}}]
Suppose $\phi: R \to R'$ is an algebra retract of standard graded $\kk$-algebras.  Then $R$ is Koszul if and only if $R'$ is Koszul and $R'$ has a linear resolution as an $R$-module (via $\phi$).
\end{prop}

By the above proposition, instead of showing that $k$ does not have a linear resolution over $R$, we need only to show that $R'$ does not have a linear resolution over $R$. 
This reduces the number of independent syzygies that we need to compute from at worst 340 (for the residue field) to at worst 80 (for $R'$).  It also has the advantage that the matrices we need to compute are generally smaller than those we would need to compute for alternative approaches such as computing the deviations of $R$ via a minimal model of $R$ over $S$; see \cite{infinite:free:resolutions} for further details.

To simplify the computation of the syzygies, we also exploit the bihomogeneous structure of $R$.

\begin{lemma} \label{bad:non-Koszul:algebra:basis}
The elements $a^ib^j, xa^i, ya^i, x^2$ for all $i, j \geq 0$ form a bihomogeneous basis for $R$ over $\kk$.
\end{lemma}

\begin{proof}
It is easily seen that $R'$ has $1, x, y, x^2$ as a basis over $\kk$ so that the polynomial ring $R'[a, b]$ has a bihomogeneous basis consisting of the elements $a^ib^j$, $xa^ib^j$, $ya^ib^j$, and $x^2a^ib^j$ for all $i, j \geq 0$.  We note that $R \iso R'[a, b]/J$, where $J = (bx, ax-by)$, and modulo $J$, we have
\vspace{-5 ex}
\begin{multicols}{2}
\begin{align*} 
xa^ib^j &\equiv 0  & (j > 0) \\[1 ex]
ya^ib^j &\equiv xa^{i+1}b^{j-1} \equiv 0 & (j > 1) \\[1 ex]
\!\!ya^ib &\equiv xa^{i+1}
\end{align*}
\vspace{0.25 em}
\begin{align*}
x^2a^ib^j &\equiv 0 & (j > 0) \\[1 ex]
x^2a^i &\equiv xya^{i-1}b \equiv 0 & (i > 0)
\end{align*}
\end{multicols}
\noindent so that $R$ is spanned by the elements  $a^ib^j, xa^i, ya^i, x^2$ for all $i, j \geq 0$.  To prove that these elements are $\kk$-linearly independent, it suffices to focus on a particular bidegree. For each $n\geq 0$, the elements of bidegree $(0,n)$ are the monomials $a^ib^{n-i}$, which are all linearly independent because they are even linearly independent in $R/(x, y) \iso \kk[a, b]$. The only monomial in the spanning set of bidegree $(2,0)$ is $x^2$, which is nonzero as it is nonzero even in $R/(a, b) \iso R'$.  It remains to see that for every $i$ the two monomials of bidegree $(1, i)$ $xa^i$ and $ya^i$ are linearly independent in $R$.

If not, there are $\alpha,\beta\in \kk$ not both zero and $f_1, f_2 \in R'[a, b]$ such that $\alpha a^i x + \beta a^i y = f_1bx + f_2(ax - by)$. We may assume $f_1, f_2$ are bihomogeneous of degree $(0, i-1)$ so that $f_1, f_2 \in \kk[a, b]$. Then $x(\alpha a^i - f_1b -f_2a) + y (f_2b - \beta a^i)=0$ in $R'[a,b]$. Since $x, y$ are part of a basis for $R'[a, b]$ over $\kk[a,b]$, it follows that $\alpha a^i - f_1b -f_2a=0$ and $f_2b-\beta a^i=0$.  From the latter, we obtain $\beta = 0 = f_2$ so that $\alpha a^i - f_1b=0$.  Consequently, we must have $\alpha = 0 = f_1$, which is a contradiction. Thus, $xa^i$ and $ya^i$ must be linearly independent in $R$ as claimed.
\end{proof}

Recall that by the above we need to resolve  the ideal $U = (a, b)$ over $R$. In the following statement, we consider the standard grading of $R$ and determine a graded free resolution of $U$, while in the proof, to simplify the arguments, we make use of the bigrading on $R$.

\begin{lemma}
The first three steps of the minimal free resolution of the ideal $U = (a, b)$ over $R$ are
\[ R(-4)^{11} \stackrel{\dd_3}{\longto} R(-3)^6 \stackrel{\dd_2}{\longto} R(-2)^3 \stackrel{\dd_1}{\longto} R(-1)^2  \longto 0,\]
where:
\[ \dd_1 = 
\left(
\begin{array}{cc;{2pt/2pt}c} 
x & 0 & b \\ 
-y & x & -a 
\end{array} 
\right),
\quad 
\dd_2 = 
\left(
\begin{array}{cc;{2pt/2pt}cccc} 
y & 0 & b & 0 & - b & -a \\ 
x & y & a & b & 0 & 0 \\  \hdashline[2pt/2pt]
0 & 0 & 0 & 0 & x & y  
\end{array}
\right), 
\]
\setcounter{MaxMatrixCols}{50}
\[ \dd_3 = 
\left(
\begin{array}{cc;{2pt/2pt}ccccccc;{2pt/2pt}cc} 

x & 0 & 0 & -b & 0 & 0 & b & b & a & 0 & 0 \\
-y & x & -b & -a & 0 & -b & 0 & 0 & -b & 0 & 0 \\ \hdashline[2pt/2pt]
0 & 0 & x & y & 0 & 0 & 0 & 0 & 0 & 0 & b \\
0 & 0 & 0 & 0 & x & y & 0 & 0 & 0 & 0 & -a \\
0 & 0 & 0 & 0 & 0 & 0 & y & 0 & -x & a & b \\
0 & 0 & 0 & 0 & 0  & 0 & 0 & x & y & -b & 0
\end{array}
\right).
\]
\end{lemma}

\begin{proof}
Since $U$ is a bihomogeneous ideal, its syzygies will also be generated by bihomogeneous elements.  At each step, we consider cases based on the first component of the bidegree of a syzygy; the number of cases will be small as $R$ is nonzero only in bidegrees $(0, i)$, $(1, i)$, and $(2, 0)$.

Let $f = (f_1, f_2) \in R(0,-1)^2$ be a bihomogeneous syzygy of $U$.  If $f$ has bidegree $(0, i+2)$, then $f_i \in \kk[a, b]$ so that $f = h(b, -a)$ for some $h \in \kk[a, b]$.  If $f$ has bidegree $(1, i+1)$, then we can write $f_j = c_{j,1}xa^i + c_{j,2}ya^i$ for some $c_{j,r} \in \kk$ so that
\[ 0 = f_1a+ f_2b = c_{1,1}xa^{i+1} + c_{1,2}ya^{i+1} + c_{2,2}ya^ib = (c_{1,1} + c_{2,2})xa^{i+1} + c_{1,2}ya^{i+1}. \]
Hence, we must have $c_{2,2} = -c_{1,1}$ and $c_{1,2} = 0$ by the above lemma so that $f = c_{1,1}a^i(x, -y) + c_{2,1}a^i(0, x)$.  If $f$ has bidegree $(2, i+1)$, then $f = (c_1x^2, c_2x^2) = c_1(x^2, 0) + c_2(0, x^2)$ for some $c_i \in \kk$, and there is no restriction on the $c_i$ since $x^2$ is in the socle of $R$.  However, we note that $(x^2, 0) = x(x, -y)$ and $(0, x^2) = x(0, x)$ belong to the span of the previously found syzygies.  Hence, the columns of $\dd_1$ span the syzygies of $U$, and since these linear syzygies are easily seen to be linearly independent over $\kk$, they form a minimal set of generators for the first syzygies.

Let $f = (f_1, f_2, f_3) \in R(-1,-1)^2 \oplus R(0, -2)$ be a bihomogeneous syzygy on the columns of $\dd_1$.  If $f$ has bidegree $(0, i + 3)$, then $f_1 = f_2 = 0$ and $f_3 \in \kk[a, b]$ so that $bf_3 = 0$ and, hence, $f_3 = 0$.  If $f$ has bidegree $(1, i+2)$, then $f_1, f_2 \in \kk[a, b]$ and $f_3 = c_{3,1}xa^i + c_{3,2}ya^i$ for some $c_{3,r} \in \kk$.   Writing $f_j = \alpha_{j,0}a^{i+1} + \alpha_{j,1}a^ib + b^2f'_j$ for some $\alpha_{j,r} \in \kk$ and $f'_j \in \kk[a, b]$, we see that $\dd_1f = 0$ implies
\begin{align*}
0 &= f_1x + f_3b = (\alpha_{1,0} + c_{3,2})xa^{i+1}, \\
0 &= -f_1y + f_2x - f_3a = (-\alpha_{1,1} + \alpha_{2,0} - c_{3,1})xa^{i+1} + (-\alpha_{1,0} - c_{3,2})ya^{i+1},
\end{align*}
so that $\alpha_{1,0} = -c_{3,2}$ and $\alpha_{1,1} = \alpha_{2,0} - c_{3,1}$.  It follows that $f$ is contained in the submodule spanned by the columns of the matrix
\[ 
M_2 = \begin{pmatrix} 
b & 0 & - b & -a \\ 
a & b & 0 & 0 \\ \hdashline[2pt/2pt]
0 & 0 & x & y
\end{pmatrix},
\]
and the additional syzygies $(b^2, 0, 0)$ and $(0, b^2, 0)$.  However, since $(b^2, 0, 0) = b(b, a, 0) - a(0, b, 0)$ and $(0, b^2, 0) = b(0, b, 0)$, these last two syzygies are not needed to generate the syzygies of bidegree $(1, i+2)$.  

If $f$ has bidegree $(2, i+1)$, then for $j = 1, 2$ we have $f_j = c_{j,1}xa^i + c_{j,2}ya^i$, and $f_3 = c_3x^2$ for some $c_3 \in \kk$ so that $0 = f_1x + f_3b = c_{1,1}x^2a^i$ and $0 = -f_1y + f_2x - f_3a = (-c_{1,2} + c_{2,1})x^2a^i$.  If $i > 0$, this imposes no restrictions on the $f_j$ so that $f$ is in the column span of the matrix
\[ \begin{pmatrix} 
xa & ya & 0 & 0 & 0 \\ 
0 & 0 & xa & ya & 0 \\
0 & 0 & 0 & 0 & x^2
\end{pmatrix}
=
M_2
\begin{pmatrix}
0 & 0 & x & y & 0 \\
0 & 0 & 0 & 0 & 0 \\
0 & x & 0 & 0 & x \\
-x & y & 0 & x & 0
\end{pmatrix},
\]
where the factorization on the right shows these syzygies are generated by the syzygies already found in bidegree $(1, 2)$.  However, if $i = 0$, we must have $c_{1, 1} = 0$ and $c_{1,2} = c_{2,1}$ so that $f = c_{1,1}(y, x, 0) + c_{2,2}(0, y, 0)$.  

Finally, if $f$ has bidegree $(3, i+1)$, then $f = (c_1x^2, c_2x^2, 0)$ for some $c_j \in \kk$, and there are no restrictions on the $c_i$.  However, the syzygies $(x^2, 0, 0) = y(y, x, 0)$ and $(0,x^2,0) = y(0, y, 0)$ are generated by those already found in bidegree $(2, 1)$.  Hence, the columns of $\dd_2$ span the syzygies of $\Im \dd_1$, and since these linear syzygies are easily seen to be linearly independent over $\kk$, they form a minimal set of generators for the second syzygies of $U$.

Let $f \in R(-2, -1)^2 \oplus R(-1, -2)^4$ be a bihomogeneous syzygy on the columns of $\dd_2$.  If $f$ has bidegree $(1, i + 3)$, then $f_1 = f_2 = 0$ and $f_j \in \kk[a, b]$ for $j \geq 3$.  Then $\dd_2f = 0$ implies that $-af_6 + (f_3 - f_5)h_2 = 0$ and $af_3 + bf_4 = 0$ so that $(-f_6, f_3 -f_5) = h_1(b, -a)$ and $(f_3, f_4) = h_2(b, -a)$ for some $h_i \in \kk[a, b]$, and the last component of $\dd_2f$ imposes no additional restrictions.  Hence, $f = h_1(0, 0, 0, 0, a, -b) + h_2(0, 0, b, -a , b, 0)$.  

If $f$ has bidegree $(2, i + 2)$, then $f_1, f_2 \in \kk[a, b]$, and for all $j \geq 3$, we have $f_j = c_{j,1}xa^i + c_{j,2}ya^i$ for some $c_{j,r} \in \kk$. 
Writing $f_j = \alpha_{j,0}a^{i+1} + \alpha_{j,1}a^ib + b^2f'_j$ for some $\alpha_{j, s} \in \kk$ and $f'_j \in \kk[a, b]$ for $j = 1, 2$, we see that $\dd_2f = 0$ implies 
\begin{align*}
0 &= yf_1 + bf_3 -b f_5 - af_6 =  (\alpha_{1,1} + c_{3,2} - c_{5,2} - c_{6,1})xa^{i+1} + (\alpha_{1,0} - c_{6,2})ya^{i+1}, \\
0 &= xf_1 + yf_2 + af_3 + bf_4 = (\alpha_{1,0} + \alpha_{2,1} + c_{3,1} + c_{4,2})xa^{i+1} + (\alpha_{2,0} + c_{3,2})ya^{i+1}, \\
0 &= xf_5 + yf_6 = (c_{5,1}+c_{6,2})x^2a^i,
\end{align*}
so that $\alpha_{1,0} = c_{6,2}$, $\alpha_{1,1} = -c_{3,2} + c_{5,2} + c_{6,1}$, $\alpha_{2,0} = -c_{3,2}$, $\alpha_{2,1} = -c_{3,1} - c_{4,2} - c_{6,2}$, and if $i = 0$ we also have $c_{5,1} = -c_{6,2}$.  Hence, if $i = 0$, we have $f'_j = 0$ for $j = 1, 2$, and we see that $f$ is contained in the column span of the matrix
\[ M_3 = 
\begin{pmatrix} 
0 & -b & 0 & 0 & b & b & a \\
-b & -a & 0 & -b & 0 & 0 & -b\\ \hdashline[2pt/2pt]
x & y & 0 & 0 & 0 & 0 & 0 \\
0 & 0 & x & y & 0 & 0 & 0 \\
0 & 0 & 0 & 0 & y & 0 & -x \\
0 & 0 & 0  & 0 & 0 & x & y
\end{pmatrix} .
\]
When $i \geq 1$, we see that $f$ is contained in the span of the columns of $M_3$ plus the four additional syzygies
\[
\begin{pmatrix} 
b^2 & 0 & 0 & a^2 \\ 
0 & b^2 & 0 & -ab \\ \hdashline[2pt/2pt]
0 & 0 & 0 & 0 \\ 
0 & 0 & 0 & 0\\ 
0 & 0 & xa & 0 \\ 
0 & 0 & 0 & ya
\end{pmatrix}
= 
M_3
\begin{pmatrix}
0 & 0 & 0  & 0 \\
0 & 0 & 0 & 0 \\
0 & a & 0 & 0 \\
0 & -b & 0 & 0 \\
0 & 0 & b & b \\
b & 0 & -b & -b \\
0 & 0 & 0 & a
\end{pmatrix},
\]
where the factorization on the right shows these syzygies are generated by the syzygies already found in bidegree $(2, 2)$.  

If $f$ has bidegree $(3, i + 1)$, then $f_j = c_{j,1}xa^i + c_{j,2}ya^i$ for some $c_{j,r} \in \kk$ for $j = 1, 2$ and $f_j = c_jx^2$ for some $c_j \in \kk$ for all $j \geq 3$.  Then $\dd_2f = 0$ implies $c_{1,2}x^2a^i = (c_{1,1} + c_{2,2})x^2a^i = 0$.  If $i = 0$, then $f = c_{1,1}(x, -y, 0, 0, 0, 0) + c_{2,1}(0, x, 0, 0, 0, 0)$, and if $i \geq 1$, there are no restrictions on $f$ so that $f$ is contained in the column span of the matrix
\[ 
\begin{pmatrix}
xa & ya & 0 & 0 & 0 & 0 & 0 & 0\\
0 & 0 & xa & ya & 0 & 0 & 0 & 0 \\ \hdashline[2pt/2pt]
0 & 0 & 0 & 0 & x^2 & 0 & 0 & 0\\
0 & 0 & 0& 0 & 0 & x^2 & 0 & 0 \\
0 & 0 & 0 & 0 & 0 & 0 & x^2 & 0 \\
0 & 0 & 0 & 0 & 0 & 0 & 0 & x^2
\end{pmatrix}
= M_3
\begin{pmatrix}
0 & -y & 0 & x & x & 0 & 0 & 0 \\
0 & 0 & -x & -y  & 0 & 0 & 0 & 0 \\
0 & 0 & 0 & 0 & 0 & x & 0 & 0 \\
0 & 0 & 0 & 0 & 0 & 0 & 0 & 0 \\
0 & 0 & 0 & 0 & 0 & 0 & y & 0 \\
y & -x & 0 & -y & 0 & 0 & -y & x \\
0 & y & 0 & 0 & 0 & 0 & 0 & 0
\end{pmatrix},
\]
where the factorization on the right shows these extra syzygies are generated by the syzygies already found in bidegree $(2, 2)$.  

Finally, if $f$ has bidegree $(4, i+1)$, then $i = 0$ and $f_j = 0$ for $j \geq 3$ so that $f$ is in the span of $(x^2, 0, 0, 0, 0, 0) = x(x, -y, 0, 0, 0, 0)$ and $(0, x^2, 0, 0, 0, 0) = x(0, x, 0, 0, 0, 0)$.  Hence, the columns of $\dd_3$ span the syzygies of $\Im \dd_2$, and since these linear syzygies are easily seen to be linearly independent over $\kk$, they form a minimal set of generators for the third syzygies of $U$.
\end{proof}

\begin{thm}
\hfill
\begin{enumerate}[label = \textnormal{(\alph*)}]
\item If $\ch(\kk) = 2$, then the map $\dd_3$ of the preceding lemma has a minimal quadratic syzygy.
\item If $\ch(\kk) \neq 2$, then the syzygies of the map $\dd_3$ of the preceding lemma are the image of the map $\dd_4: R(-5)^{20} \to R(-4)^{11}$ given by
\begin{center}
\scalebox{0.85}{
$\dd_4 = 
\left(
\begin{array}{cc;{2pt/2pt}cccccccccc;{2pt/2pt}cccccccc}
y & 0 & b & 0 & -b & -b & -a & 0 & 0 & -b & 0 & 0 & 0 & 0 & 0 & 0 & 0 & 0 & 0 & 0 \\
x & y & a & b & 0 & 0 & 0 & 0 & 0 & 0 & 0 & 0 & 0 & 0 & 0 & 0 & 0 & 0 & 0 & 0 \\ \hdashline[2pt/2pt]
0 & 0 & 0 & 0 & y & 0 & -x & 0 & 0 & 0 & 0 & 0 & 0 & -b & -a & -b & a & 0 & 0 & 0 \\
0 & 0 & 0 & 0 & 0 & x & y & 0 & 0 & 0 & 0 & 0  & 0 & 0 & b & 0 & -b & 0 & 0 & -b \\
0 & 0 & 0 & 0 & 0 & 0 & 0 & y & 0 & -x & 0 & 0 & b & -a & 0 & 0 & 0 & 0 & a & 0 \\
0 & 0 & 0 & 0 & 0 & 0 & 0 & 0 & x & y & 0 & 0 & 0 & b & 0 & 0 & 0 & 0 & 0 & a \\
0 & 0 & 0 & 0 & 0 & 0 & 0 & 0 & 0 & 0 & x & y & 0 & 0 & 0 & 0 & -b & -a & 0 & -b \\
0 & 0 & 0 & 0 & 0 & 0 & 2y & 0 & 0 & 0 & 0 & -2y & 0 & 0 & b & -a & 0 & a & 0 & 0 \\
0 & 0 & 0 & 0 & 0 & 0 & 0 & 0 & 0 & 0 & 0 & x & 0 & 0 & 0 & b & 0 & 0 & 0 & 0 \\ \hdashline[2pt/2pt]
0 & 0 & 0 & 0 & 0 & 0 & 0 & 0 & 0 & 0 & 0 & 0 & 0 & 0 & 0 & 0 & x & y & 0 & 0 \\
0 & 0 & 0 & 0 & 0 & 0 & 0 & 0 & 0 &0  & 0 & 0 & 0 & 0 & 0 & 0 & 0 & 0 & x & y
\end{array}
\right),
$
}
\end{center}
and $\dd_4$ has a minimal quadratic syzygy.
\end{enumerate}
Hence, in either case, $R$ is not Koszul.
\end{thm}

\begin{proof}
Let $f \in R(-3, -1)^2 \oplus R(-2, -2)^7 \oplus R(-1, -3)^2$  be a bihomogeneous syzygy on the columns of $\dd_3$.  If $f$ has bidegree $(1, i + 4)$, then $f_j = 0$ for $j \leq 9$ and $f_{10}, f_{11} \in \kk[a, b]$.  Since $\dd_3f = 0$ implies $bf_{10} = bf_{11} = 0$, it follows that $f = 0$.  

If $f$ has bidegree $(2, i + 3)$, then $f_1 = f_2 = 0$, $f_j \in \kk[a, b]$ for $3 \leq j \leq 9$, and $f_j = c_{j,1}xa^i + c_{j,2}ya^i$ for some $c_{j, r} \in \kk$ for $j = 10, 11$.  First, we note that $\dd_3f = 0$ implies $af_9 + b(-f_4 + f_7 + f_8) = 0$ and $af_4 + b(f_3 + f_6 + f_9) = 0$ so that there exist $h_1, h_2 \in \kk[a, b]$ such that: 
\[
\begin{pmatrix} f_9 \\ -f_4 + f_7 + f_8 \end{pmatrix} = h_1\begin{pmatrix} b \\ - a\end{pmatrix} \qquad 
\begin{pmatrix} f_4 \\ f_3 + f_6 + f_9 \end{pmatrix} = h_1\begin{pmatrix} b \\ - a\end{pmatrix}.
\]
Writing $f_j = \alpha_{j,0}a^{i+1} + \alpha_{j,1}a^ib + b^2f'_j$ for some $\alpha_{j, s} \in \kk$ and $f'_j \in \kk[a, b]$, this forces $h_1 = \alpha_{9,1}a^i + bf'_9$, $h_2 = \alpha_{4,1}a^i + bf'_4$, and:  
\begin{align*}
0 &= \alpha_{4,0} = \alpha_{9, 0}, \\
0 &= -f_4 + f_7 + f_8 + ah_1 = (\alpha_{7,0} + \alpha_{8,0} + \alpha_{9,1})a^{i+1} + (-\alpha_{4,1} + \alpha_{7,1} + \alpha_{8,1})a^ib, \\
& \hspace{12 em} + (-f'_4 + f'_7 + f'_8)b^2 + abf'_9, \\
0 &= f_3 + f_6 + f_9 + ah_2 = (\alpha_{3,0} + \alpha_{4,1} + \alpha_{6,0})a^{i+1} + (\alpha_{3,1} + \alpha_{6,1} + \alpha_{9,1})a^ib,  \\
& \hspace{12 em} + (f'_3 + f'_6 + f'_9)b^2 + abf'_4.
\end{align*}
Consequently, the remaining components of $\dd_3f$ imply that
\begin{align*}
0 &= xf_3 + yf_4 + bf_{11} =  (\alpha_{3,0} + \alpha_{4,1} + c_{11,2})xa^{i+1}, \\
0 &= xf_5 + yf_6 - af_{11} = (\alpha_{5,0} + \alpha_{6,1} -c_{11,1})xa^{i+1} + (\alpha_{6,0} - c_{11,2})ya^{i+1}, \\
0 &= yf_7 - xf_9 +af_{10} + bf_{11} = (\alpha_{7,1} + c_{10,1} + c_{11,2})xa^{i+1} + (\alpha_{7,0} + c_{10,2})ya^{i+1}, \\
0 &= xf_8 + yf_9 - bf_{10} = (\alpha_{8,0} + \alpha_{9,1} - c_{10,2})xa^{i+1}, 
\end{align*}
so that we have the linear equalities: 

\begin{minipage}{\textwidth}
\begin{center}
\begin{minipage}{0.3\textwidth}
\begin{align*}
\alpha_{3,0} &= -\alpha_{4,1} - c_{11,2}, \\
\alpha_{4,0} &= 0, \\
\alpha_{5,0} &= -\alpha_{6,1} + c_{11,1},
\end{align*}
\end{minipage}
\begin{minipage}{0.3\textwidth}
\begin{align*}
\alpha_{6,0} &= c_{11,2}, \\
\alpha_{7,0} &= -c_{10,2}, \\
\alpha_{7,1} &= -c_{10,1} -c_{11,2}, 
\end{align*}
\end{minipage}
\begin{minipage}{0.3\textwidth}
\begin{align*}
\alpha_{8,0} &= -\alpha_{9,1} + c_{10,2},\\
\alpha_{9,0} &= 0. 
\end{align*}
\end{minipage}
\end{center} 
\end{minipage}
\vspace{2 ex}

\noindent Moreover, if $i = 0$, then $f'_j = 0$ for all $j$ so that  we have the additional equalities 
\[ \alpha_{3,1} = -\alpha_{6,1} - \alpha_{9,1} \qquad  \alpha_{4,1} =  \alpha_{7,1}  + \alpha_{8,1} = \alpha_{8,1} - c_{10,1} - c_{11,2} \] 
so that $\alpha_{3,0} = -\alpha_{8,1} + c_{10,1}$ and $f$ is in the column span of the matrix:
\begin{center}
\scalebox{0.925}{
$
M_4 = \begin{pmatrix}
0 & 0 & 0 & 0 & 0 & 0 & 0 & 0 \\
0 & 0 & 0 & 0 & 0 & 0 & 0 & 0 \\  \hdashline[2pt/2pt]
0 & -b & -a & -b & a & 0 & 0 & 0 \\
0 & 0 & b & 0 & -b & 0 & 0 & -b \\
b & -a & 0 & 0 & 0 & 0 & a & 0 \\
0 & b & 0 & 0 & 0 & 0 & 0 & a \\
0 & 0 & 0 & 0 & -b & -a & 0 & -b \\
0 & 0 & b & -a & 0 & a & 0 & 0 \\
0 & 0 & 0 & b & 0 & 0 & 0 & 0 \\ \hdashline[2pt/2pt]
0 & 0 & 0 & 0 & x & y & 0 & 0 \\
0 & 0 & 0 & 0 & 0 & 0 & x & y
\end{pmatrix}.
$
}
\end{center}
\vspace{1 ex}
On the other hand, if $i \geq 1$, then we can write $f'_j = \alpha_{j,2}a^{i-1} + bf''_j$ for some $f''_j \in \kk[a, b]$  for $j = 4, 9$ so that: 
\[ \alpha_{3,1} = -\alpha_{6,1} - \alpha_{9,1} - \alpha_{4,2}, \qquad  \alpha_{4,1} = \alpha_{8,1} + \alpha_{9,2} - c_{10,1} - c_{11,2}, \]  
\[ f'_3 = - af''_4 - f_6 - \alpha_{9,2}a^{i-1} - bf''_9, \qquad f'_7 = \alpha_{4,2}a^{i-1}  + bf''_4 -f'_8 - af''_9. \]
In this case, $\alpha_{3,0} = -\alpha_{8,1} - \alpha_{9,2} + c_{10,1}$ so that $f$ is contained in the submodule spanned by the columns of $M_4$ and the matrix \vspace{1 ex}

\begin{minipage}{\textwidth}
\begin{center}
\scalebox{0.875}{
$
\begin{pmatrix}
0 & 0 & 0 & 0 & 0 & 0 \\
0 & 0 & 0 & 0 & 0 & 0 \\ \hdashline[2pt/2pt]
-ab & 0 & -b^2 & 0 & -a^2-b^2 &-b^3  \\
b^2 & 0 & 0 & 0 & ab & 0 \\
0 & b^2 & 0 & 0 & 0 & 0 \\
0 & 0 & b^2 & 0 & 0 & 0\\
b^2 & 0 & 0 & -b^2 & 0 & -ab^2 \\
0 & 0 & 0 & b^2 & 0 & 0\\
0 & 0 & 0 & 0 & b^2 & b^3 \\ \hdashline[2pt/2pt]
0 & 0 & 0 & 0 & 0 & 0 \\
0 & 0 & 0 & 0 & 0 & 0
\end{pmatrix}
= M_4
\begin{pmatrix}
0 & b & 0 & 0 & 0 & 0 \\
0 & 0 & b & 0 & 0 & 0 \\
0 & 0 & 0 & b & a & 0 \\
0 & 0 & 0 & 0 & b & b^2 \\
-b & 0 & 0 & b & 0 & 0 \\
0 & 0 & 0 & 0 & 0 & b^2 \\
0 & 0 & b & 0 & 0 & 0 \\
0 & 0 & 0 & 0 & 0 & 0
\end{pmatrix},
$
}
\end{center}
\vspace{1 ex}
\end{minipage}

\noindent where the factorization on the right shows these extra syzygies are generated by the syzygies already found in bidegree $(2, 3)$.  

If $f$ has bidegree $(3, i + 2)$, then $f_1, f_2 \in \kk[a, b]$, $f_j = c_{j,1}xa^i + c_{j,2}ya^i$ for some $c_{j,r} \in \kk$ for $3 \leq j \leq 9$, and $f_j = c_jx^2$ for some $c_j \in \kk$ for $j = 10, 11$.  Writing $f_j = \alpha_{j,0}a^{i+1} + \alpha_{j,1}a^ib + b^2f'_j$ for some $f'_j \in \kk[a, b]$ for $j = 1, 2$, we see that $\dd_3f = 0$ implies that
\begin{align*}
0 &= xf_1 - bf_4 + bf_7 + bf_8 + af_9, \\
&= (\alpha_{1,0} - c_{4,2} + c_{7,2} + c_{8,2} + c_{9,1})xa^{i+1} + c_{9,2}ya^{i+1}, \\
0 &= -yf_1 + xf_2 -bf_3 - af_4 - bf_6 - bf_9, \\
&= (-\alpha_{1,1} + \alpha_{2,0} - c_{3,2} - c_{4,1} - c_{6,2} - c_{9,2})xa^{i+1} + (-\alpha_{1,0} - c_{4,2})ya^{i+1}, \\
0 &= xf_3 + yf_4 + bf_{10} = (c_{3,1} + c_{4,2})x^2a^i, \\
0 &= xf_5 + yf_6 - af_{11} = (c_{5,1} + c_{6,2})x^2a^i, \\
0 &= yf_7 - xf_9 + af_{10} + bf_{11} = (c_{7,2} - c_{9,1})x^2a^i, \\
0 &= xf_8 + yf_9 - bf_{10} = (c_{8,1} + c_{9,2})x^2a^i,
\end{align*}
so that we have the linear equalities: 

\begin{minipage}{0.9\textwidth}
\begin{center}
\begin{minipage}{0.4\textwidth}
\begin{align*}
\alpha_{1,0} &= -c_{4,2}, \\
\alpha_{1,1} &= \alpha_{2,0} - c_{3,2} - c_{4,1} - c_{6,2},
\end{align*}
\end{minipage}
\begin{minipage}{0.4\textwidth}
\begin{align*}
c_{7,2} &= 2c_{4,2} - c_{8,2} - c_{9,1}, \\
c_{9,2} &= 0.
\end{align*}
\end{minipage}
\end{center} 
\end{minipage}
\vspace{2 ex}

\noindent In addition, if $i = 0$, then $f'_1 = f'_2 = 0$ and $c_{10} = c_{11} = 0$ so that we have the additional equalities
\[ c_{3,1} = -c_{4,2} \qquad c_{5,1} = -c_{6,2} \qquad c_{7,2} = c_{9,1} \qquad c_{8,1} = 0, \]
so that $c_{8,2} = 2c_{4,2} - 2c_{9,1}$.  Hence, $f$ is in the column span of the matrix:
\begin{center}
\scalebox{0.9}{
$N_4 =
\begin{pmatrix}
b & 0 & -b & -b & -a & 0 & 0 & -b & 0 & 0 \\
a & b & 0 & 0 & 0 & 0 & 0 & 0 & 0 & 0 \\ \hdashline[2pt/2pt]
0 & 0 & y & 0 & -x & 0 & 0 & 0 & 0 & 0 \\
0 & 0 & 0 & x & y & 0 & 0 & 0 & 0 & 0  \\
0 & 0 & 0 & 0 & 0 & y & 0 & -x & 0 & 0 \\
0 & 0 & 0 & 0 & 0 & 0 & x & y & 0 & 0 \\
0 & 0 & 0 & 0 & 0 & 0 & 0 & 0 & x & y  \\
0 & 0 & 0 & 0 & 2y & 0 & 0 & 0 & 0 & -2y \\
0 & 0 & 0 & 0 & 0 & 0 & 0 & 0 & 0 & x \\ \hdashline[2pt/2pt]
0 & 0 & 0 & 0 & 0 & 0 & 0 & 0 & 0 & 0 \\
0 & 0 & 0 & 0 & 0 & 0 & 0 &0  & 0 & 0
\end{pmatrix}.
$
}
\end{center}
On the other hand, if $i \geq 1$, then $f$ is in the span of the columns of $N_4$ and the columns of the matrix
\begin{center}
\vspace{-2 ex}
\scalebox{0.82}{
$
\begin{pmatrix}
b^2 & 0 & 0 & -a^2 & 0 & -ab & 0 & 0 & 0 \\
0 & b^2 & 0 & 0 & 0 & 0 & 0 & 0 & 0 \\ \hdashline[2pt/2pt]
0 & 0 & xa & 0 & 0 & 0 & 0 & 0 & 0  \\
0 & 0 & 0 &ya  & 0 & 0 & 0 & 0 & 0 \\
0 & 0 & 0 & 0 & xa & 0 & 0 & 0 & 0 \\
0 & 0 & 0 & 0  & 0 & ya & 0 & 0 & 0  \\
0 & 0 & 0 & 2ya & 0 & 0 & 0 & -ya & -ya \\
0 & 0 & 0 & 0 & 0 & 0 & xa & ya & 0 \\
0 & 0 & 0 & 0 & 0 & 0 & 0 & 0 & xa \\ \hdashline[2pt/2pt]
0 & 0 &  0 & 0 & 0 & 0 & 0 & 0 & 0 \\
0 & 0 &  0 & 0 & 0 & 0 & 0 & 0 & 0
\end{pmatrix}
= \left(\begin{array}{c;{2pt/2pt}c} N_4 & M_4 \end{array}\right)
\left(
\begin{array}{ccccccccc}
0 & 0 & 0 & 0 & 0 & 0 & 0 & 0 \\
0 & b & 0 & 0 & 0 & 0 & 0 & 0 \\
0 & 0 & b & 0 & 0 & 0 & 0 & 0 \\
-b & 0 & -b & 0 & 0 & 0 & 0 & 0 \\
0 & 0 & 0 & a & 0 & 0 & 0 & 0 \\
0 & 0 & 0 & 0 & b & b & 0 & 0 \\
0 & 0 & 0 & 0 & 0 & 0 & 0 & 0 \\
0 & 0 & 0 & 0 & 0 & a & 0 & 0  \\
0 & 0 & 0 & 0 & 0 & 0 & 0 & 0 \\
0 & 0 & 0 & a & 0 & 0 & 0 & a \\ \hdashline[2pt/2pt]
0 & 0 & 0 & 0 & 0 & 0 & 0 & 0 \\
0 & 0 & 0 & 0 & 0 & 0 & 0 & 0 \\
0 & 0 & 0 & x & 0 & 0 & -x & -2x \\
0 & 0 & 0 & -y & 0 & 0 & 0 & 0 \\
0 & 0 & 0  & x & 0 & 0 & -x & -2x \\
0 & 0 & 0 & -y & 0 & 0 & y & 2y\\
0 & 0 & 0 & 0 & 0 & 0 & 0 & 0 \\
0 & 0 & 0 & 0 & 0 & 0 & 0 & 0
\end{array} 
\right),
$
}
\end{center}
where the factorization on the right shows these extra syzygies are generated by the syzygies already found in bidegrees $(3, 2)$ and $(2, 3)$.

If $f$ has bidegree $(4, i+1)$, then we can write $f_j = c_{j,1}xa^i + c_{j,2}ya^i$ for some $c_{j, r} \in \kk$ for $j = 1, 2$, $f_j = c_jx^2$ for some $c_j \in \kk$ for $3 \leq j \leq 9$, and $f_{10} = f_{11} = 0$.  Then $\dd_3f = 0$ implies $0 = xf_1 = c_{1,1}x^2a^i$ and $0 = -yf_1 + xf_2 = (c_{1,2} - c_{2,1})x^2a^i$.  If $i = 0$, then $c_{1,1} = 0$, $c_{1, 2} = c_{2,1}$, and $c_j = 0$ for $3 \leq j \leq 9$ so that $f = c_{2,1}(y, x) + c_{2,2}(0, y) \in R(-3,-1)^2$.  On the other hand, if $i \geq 1$, then there are no restrictions on the $c_{j, r}$ or $c_j$ so that $f$ is contained in the submodule of $R(-3,-1)^2 \oplus R(-2,-2)^7$ spanned by the columns of the matrix below and the syzygy $s = (0, 0, 0, 0, 0, 0, 0, x^2, 0)$.
\begin{center}
\scalebox{0.85}{
$
\begin{pmatrix}
xa & ya & 0 & 0 & 0 & 0  & 0 & 0 & 0 & 0 \\
0 & 0 & xa & ya & 0 & 0 & 0 & 0 & 0 & 0 \\
0 & 0 & 0 & 0 & x^2 & 0 & 0 & 0 & 0 & 0 \\
0 & 0 & 0 & 0 & 0 & x^2 & 0 & 0 & 0 & 0 \\
0 & 0 & 0 & 0 & 0 & 0 & x^2 & 0 & 0 & 0 \\
0 & 0 & 0 & 0 & 0 & 0 & 0 & x^2 & 0 & 0 \\
0 & 0 & 0 & 0 & 0 & 0 & 0 & 0 & x^2 & 0 \\
0 & 0 & 0 & 0 & 0 & 0 & 0 & 0 & 0 & 0 \\
0 & 0 & 0 & 0 & 0 & 0 & 0 & 0 & 0 & x^2 \\
\end{pmatrix}
= N_4
\begin{pmatrix}
0 & 0 & 0 & y  & 0 & 0 & 0 & 0 & 0 & 0 \\
0 & 0 & y & 0 & 0 & 0 & 0 & 0 & 0 & 0 \\
0 & 0 & 0 & 0 & y & 0 & 0 & 0 & 0 & 0 \\
0 &  x & 0 & 0 & -y & x & 0 & 0 & 0 & 0 \\
0 & -y & 0 & 0 & 0 & 0 & 0 & 0 & 0 & 0 \\
0 & 0 & 0 & 0 & 0 & 0 & y & 0 & 0 & 0 \\
x & 0 & 0 & -x & 0 & 0 & 0 & x & 0 & 0 \\
-y & 0 & 0 & y & 0 & 0 & 0  & 0 & 0 & 0 \\
0 & x & 0 & 0 & 0 & 0 & 0 & 0 & x & 0 \\
0 & -y & 0 & 0 & 0 & 0 & 0 & 0 & 0 & x
\end{pmatrix}.
$
}
\end{center}
The factorization on the right above shows all of these extra syzygies except possibly $s$ are generated by the syzygies already found in bidegree $(3, 2)$. If $\ch(\kk) \neq 2$, then $s = N_4v$ where $v = (0, 0, 0, 0, 0, 0, 0, 0, \frac{1}{2}x, -\frac{1}{2}y)$ so that $s$ is also redundant.  However, if $\ch(\kk) = 2$, it is easily seen that $s = \left(\begin{array}{c;{2pt/2pt}c} N_4 & M_4 \end{array}\right)v$ has no bihomogeneous solutions since for every $v \in R(-3, -2)^{10} \oplus R(-2, -3)^8$ of bidegree $(4, 2)$ the eighth coordinate of $\left(\begin{array}{c;{2pt/2pt}c} N_4 & M_4 \end{array}\right)v$ is zero.  Thus, $s$ is a minimal quadratic syzygy of $\dd_3$, which proves part (a), and so, we may assume that $\ch(\kk) \neq 2$ for the remainder of the proof.

If $f$ has bidegree $(5, i + 1)$, then $i = 0$ and $f_j = 0$ for $j \geq 3$ so that $f \in R(-3, -1)^2$, and $f$ is contained is in the span of $(x^2, 0) = y(y, x)$ and $(0, x^2) = y(0, y)$ and, therefore, in the span of the syzygies already found in bidegree $(4, 1)$.  Hence, the columns of $\dd_4$ span the syzygies of $\Im \dd_3$, and since these syzygies are easily seen to be linearly independent over $\kk$, they form a minimal set of generators for the fourth syzygies of $U$.

To complete the proof of the theorem, we will show that $s = (0, \dots, 0, x^2) \in R(-4,-1)^2 \oplus R(-3, -2)^{10}$ is a minimal quadratic syzygy of bidegree $(5, 2)$ on the columns of $\dd_4$.  If $s$ were not minimal, then it would be contained in the span of the syzygies of bidegrees $(5, 1)$ and $(4, 2)$.  We will show that this is not the case.  As above, the minimal syzygies of bidegree $(5, 1)$ are easily seen to be $(x, -y), (0, x) \in R(-5, -1)^2$.  If $f \in R(-4,-1)^2 \oplus R(-3, -2)^{10}$ is a syzygy of bidegree $(4, 2)$, then $f_j = \alpha_{j,0}a + \alpha_{j,1}b$ for some $\alpha_{j, s} \in \kk$ for $j = 1, 2$ and $f_j = c_{j,1}x + c_{j,2}y$ for some $c_{j,r} \in \kk$ for $j \geq 3$.  It follows from $\dd_4f = 0$ that
\begin{align*}
0 &= yf_1 + b(f_3 - f_5 - f_6 - f_{10}) - af_7, \\
&= (\alpha_{1,1} + c_{3,2} - c_{5,2} - c_{6,2} - c_{7,1} - c_{10,2})xa + (\alpha_{1,0} - c_{7,2})ya, \\
0 &= xf_1 + yf_2 + af_3 + bf_4 = (\alpha_{1,0} + \alpha_{2,1} + c_{3,1} + c_{4,2})xa + (\alpha_{2,0} + c_{3,2})ya, \\
0 &= xf_j + yf_{j+1} = (c_{j,1} + c_{j+1,2})x^2, \qquad j \in \{6, 9, 11\},
\end{align*}
\begin{minipage}{\textwidth}
\begin{center}
\begin{minipage}{0.4\textwidth}
\begin{align*}
0 &= yf_5 - xf_7 = (c_{5,2} - c_{7,1})x^2, \\
0 &= yf_8 - xf_{10} = (c_{8,2} - c_{10,1})x^2,
\end{align*}
\end{minipage}
\hspace{1 em}
\begin{minipage}{0.4\textwidth}
\begin{align*}
0 &= 2yf_7 - 2yf_{12} = (2c_{7,2} - 2c_{12,2})x^2, \\
0 &= xf_{12} = c_{12,1}x^2,
\end{align*}
\end{minipage}
\end{center}
\vspace{1 ex}
\end{minipage}

\noindent so that

\begin{minipage}{\textwidth}
\begin{center}
\begin{minipage}{0.25\textwidth}
\begin{align*}
\alpha_{1,0} &= c_{7,2} = c_{12, 2}, \\
\alpha_{1, 1} &= -c_{3,2} + c_{6,2} + 2c_{7,1} + c_{10,2}, \\
\alpha_{2,0} &= -c_{3,2}, \\
\alpha_{2,1} &= -c_{3,1} - c_{4,2} -c_{12,2},
\end{align*}
\end{minipage}
\hspace{0.5 em}
\begin{minipage}{0.25\textwidth}
\begin{align*}
c_{5,2} &= c_{7,1}, \\
c_{6,1} &= -c_{12, 2}, \\
c_{8,2} &= c_{10,1},
\end{align*}
\end{minipage}
\hspace{0.5 em}
\begin{minipage}{0.25\textwidth}
\begin{align*}
c_{9,1} &= -c_{10,2}, \\
c_{11,1} &= -c_{12, 2}, \\
c_{12,1} &= 0,
\end{align*}
\end{minipage}
\end{center}
\vspace{1 ex}
\end{minipage}

\noindent and $f$ is contained in the column span of the matrix
\begin{center}
\scalebox{0.9}{
$
\begin{pmatrix}
0 & -b & 0 & 0 & 0 & b & 2b & 0 & 0 & 0 & b & 0 & a \\
-b & -a & 0 & -b & 0 & 0 & 0 & 0 & 0 & 0 & 0 & 0 & -b \\ \hdashline[2pt/2pt]
x & y & 0 & 0 & 0 & 0 & 0 & 0 & 0 & 0 & 0 & 0 & 0 \\
0 & 0 & x & y & 0 & 0 & 0 & 0 & 0 & 0 & 0 & 0 & 0 \\
0 & 0 & 0 & 0 & x & 0 & y & 0 & 0 & 0 & 0 & 0 & 0 \\
0 & 0 & 0 & 0 & 0 & y & 0 & 0 & 0 & 0 & 0 & 0 & -x \\
0 & 0 & 0 & 0 & 0 & 0 & x & 0 & 0 & 0 & 0 & 0 & y \\
0 & 0 & 0 & 0 & 0 & 0 & 0 & x & 0 & y & 0 & 0 & 0 \\
0 & 0 & 0 & 0 & 0 & 0 & 0 & 0 & y & 0 & -x & 0 & 0 \\
0 & 0 & 0 & 0 & 0 & 0 & 0 & 0 & 0 & x & y & 0 & 0 \\
0 & 0 & 0 & 0 & 0 & 0 & 0 & 0 & 0 & 0 & 0 & y & -x \\
0 & 0 & 0 & 0 & 0 & 0 & 0 & 0 & 0 & 0 & 0 & 0 & y \\
\end{pmatrix}.
$
}
\end{center}
If $Q$ denotes the matrix whose columns are the syzygies of bidegree $(5, 1)$ followed by the above matrix, it is easily seen that $s = Qv$ has no homogeneous solution $v \in R(-5, -1)^2 \oplus R(-4, -2)^{13}$ since the last coordinate of $v$ would necessarily be a linear form of the form $y + cx$ for some $c \in \kk$, but then the first coordinate of $Qv$ must contain $ya$ and, therefore, cannot be zero.  Thus, $s$ is a minimal quadratic 5th syzygy of $U$ so that $R' \iso R/U$ does not have a linear resolution over $R$, and $R$ is not Koszul.
\end{proof}

\end{appendix}

\end{spacing}

\end{document}